\documentclass[12pt]{amsart}

\usepackage{fullpage}
\usepackage{xypic,hyperref,todonotes,comment}
\usepackage{graphics,amssymb,mathtools}
\usepackage{amsmath,amsthm,amssymb,mathrsfs}
\usepackage[matrix,arrow,curve]{xy}
\usepackage{verbatim}
\usepackage{url}
\usepackage{hyperref}
\hypersetup{ 
colorlinks,
citecolor=black,
filecolor=blue,
linkcolor=black,
urlcolor=magenta
}
\usepackage{wasysym}
\usepackage[T1]{fontenc}

\numberwithin{equation}{subsection}

\newtheorem{Lemma}[equation]{Lemma}
\newtheorem{lemma}[equation]{Lemma}
\newtheorem{prop}[equation]{Proposition}
\newtheorem{theorem}[equation]{Theorem}
\newtheorem{coro}[equation]{Corollary}

\theoremstyle{definition}
\newtheorem{Def}[equation]{Definition}
\newtheorem{example}[equation]{Example}
\newtheorem{rema}[equation]{Remark}

\newcommand{\msc}[1]{\mathscr{{#1}}}

\newcommand{\mrm}[1]{\mathrm{#1}}
\newcommand{\mfk}[1]{\mathfrak{#1}}

\newcommand{\Hom}{{\rm Hom}}

\newcommand{\End}{{\rm End}}

\newcommand{\Inn}{{\rm Inn}}
\newcommand{\Der}{{\rm Der}}
\newcommand{\DerC}{{{\rm Der}}_{{\mathscr{Q}}}(\Omega C)}
\newcommand{\DerP}{{{\rm Der}}_{{\mathscr{Q}}}(\Omega P)}
\newcommand{\Coder}{{\rm Coder}}

\newcommand{\Ext}{{\rm Ext}}
\newcommand{\HH}{{\rm HH}}

\newcommand{\bZ}{\mathbb{Z}}

\renewcommand{\le}{\leqslant}
\renewcommand{\ge}{\geqslant}

\newcommand{\ot}{\otimes}
\newcommand{\ox}{\otimes}
\newcommand{\kk}{\mathbf{k}}

\begin{document}
\title{$A_{\infty}$-coderivations and the Gerstenhaber bracket \\ on  Hochschild cohomology}
\author{C.~Negron}
\address{Department of Mathematics\\ University of North Carolina\\ Chapel Hill, NC 27599, USA}
\email{cnegron@email.unc.edu}

\author{Y.~Volkov}
\address{Department of Mathematics and Mechanics, Saint Petersburg State University, Saint Petersburg, Russia}
\email{wolf86\_666@list.ru}

\author{S.~Witherspoon}
\address{Department of Mathematics\\Texas A\&M University\\College Station, TX 77843,
USA}
\email{sjw@math.tamu.edu}

\date{7 September 2019}

\maketitle

\begin{abstract}
We show that 
Hochschild cohomology of an algebra over a field is a space
of infinity coderivations on an arbitrary projective bimodule
resolution of the algebra.
The Gerstenhaber bracket is the graded commutator of infinity coderivations.
We thus generalize, to an arbitrary resolution,
Stasheff's realization of the Gerstenhaber bracket on Hochschild
cohomology as the graded commutator of coderivations on the tensor
coalgebra of the algebra. 
\end{abstract}

\section{Introduction}

Stasheff~\cite{Stasheff} identified the Hochschild cohomology of an algebra $A$ with a space of coderivations on the tensor coalgebra of $A$ (cf.~\cite[Part II]{Quillen}). 
In the same article the author also gave an elegant interpretation of the
Gerstenhaber bracket as the graded commutator of these coderivations,
thus providing a theoretical framework for the explicit formula
of the bracket as given by Gerstenhaber~\cite{Gerstenhaber}
on the bar resolution of $A$. 

In this paper we generalize Stasheff's result to an arbitrary 
bimodule resolution $P$ of $A$ (see Theorem~\ref{mainthm} and Section~\ref{subsec:Stasheff}). 
In order to do this, we must consider higher operations on the resolution,
starting with a
diagonal map for $P$, that is, a chain map from $P$ to $P\ot_A P$. 
If $P$ is the bar resolution, the standard diagonal map
is coassociative.
In general a diagonal map is only coassociative up to homotopy,
and this is where higher operations begin to appear.
We show that $P$ may be given the structure of an $A_{\infty}$-coalgebra.
We show that Hochschild cohomology of $A$ may 
then be identified with the graded Lie algebra of outer $A_\infty$-coderivations on $P$.  
That is, it may be identified with the quotient of the 
space of $A_{\infty}$-coderivations on $P$
by the space of inner $A_{\infty}$-coderivations, and 
the Gerstenhaber bracket is induced by 
the graded commutator of $A_{\infty}$-coderivations.  
Thus the present paper contributes to a growing collection of literature concerning alternate interpretations of the Lie structure on Hochschild 
cohomology~\cite{hermann16,Keller-derived,Negron-Witherspoon,schwede98,Stasheff,suarezalvarez17,Volkov}, 
and subsequent applications (e.g.~\cite{Grimley,GNW,hermann16II,Keller-derived,meineletal18,negronwitherspoon17}).  For related constructions one may see also recent works of Herscovich and Tamaroff~\cite{Herscovich,Tamaroff,TamaroffII}.

We note that even when the higher structure on $P$ vanishes, i.e.\ when the $A_\infty$-coalgebra structure on $P$ simply reduces to a dg coalgebra structure, the space of $A_\infty$-coderivations on $P$ is strictly larger than the space of ordinary coderivations.  This occurs, for example, when $A$ is Koszul and $P$ is the Koszul resolution.  In that case there is a canonical dg coalgebra structure on $P$ induced by the algebra structure on the Koszul dual $A^!$, and yet it is strictly necessary to consider $A_\infty$-coderivations of $P$ in order to obtain the identification with Hochschild cohomology.

Many of the proofs in the present work are constructive, and the constructions of 
$A_{\infty}$-coderivations begin with the homotopy liftings
introduced by the second author~\cite{Volkov} to generalize
work of the first and third authors~\cite{Negron-Witherspoon}:
Given a Hochschild cocycle $f$ on $P$, we show how to extend it
to an $A_{\infty}$-coderivation $\{f_n\}$ for which $f_0=f$ and $f_1$
is a homotopy lifting of $f$ as defined in~\cite{Volkov}.
If one is interested only in expressing the Gerstenhaber bracket
of two cocycles on $P$ as another cocycle on $P$, one needs only
consider homotopy liftings, and not the full 
structure of infinity coderivations.
We include some small examples as illustrations;
other examples are given in~\cite{Volkov}.

The order of the main definitions and results is as follows.
In Section~\ref{sec:prelim}, we give general definitions of
$A_{\infty}$-coalgebras, $A_{\infty}$-coderivations, and a cup product. 
In Section~\ref{sec:proj-res}, we show that a projective bimodule
resolution $P$ of an algebra $A$ has an $A_{\infty}$-coalgebra structure
(Theorem~\ref{Acostruc}).
In Section~\ref{sec:main} we prove our main theorems:
For any Hochschild cocycle $f$ on $P$, there exists an
$A_{\infty}$-coderivation on $P$ whose 0-component is $f$ (Theorem~\ref{lift_exist}) and whose 1-component is a homotopy lifting for $f$.
% (Lemma~\ref{1comp_homlif}).
As a consequence, 
Hochschild cohomology of $A$ is isomorphic to the Gerstenhaber algebra
of $A_{\infty}$-coderivations on $P$ modulo inner $A_{\infty}$-coderivations
(Theorem~\ref{mainthm}).
In Section~\ref{sec:conn-gr} we consider connected graded algebras
and in particular look closely at Koszul algebras as
one large class of examples.
Their Koszul resolutions have rich arrays of $A_{\infty}$-coderivations
for which we give explicit formulas.

For all of our definitions and results, we choose grading and
sign conventions to simplify computations.
In particular, we adopt the Koszul sign convention: 
If $f:M_1\rightarrow N_1$ and $g:M_2\rightarrow N_2$ are homogeneous morphisms of $\mathbb{Z}$-graded $A^e$-modules, then $f\ot_Ag:M_1\ot_A M_2\rightarrow N_1\ot_A N_2$ denotes the morphism defined by the equality 
\[
(f\ot_Ag)(x\ot_Ay)=(-1)^{|g||x|}f(x)\ot_Ag(y)
\]
for all homogeneous $x\in M_1, y\in M_2$.
As a consequence, $(1\ot_A g)(f\ot_A 1) = (-1)^{|f||g|} (f\ot_A g)$,
where $1$ always denotes an appropriate identity map.  For complexes of bimodules $X$ and $Y$, we let $\partial$ denote the differential on the Hom complex $\Hom_{A^e}(X,Y)$, $\partial(f)=d_Yf-(-1)^{|f|}fd_X$.

Throughout, $\kk$ will be a field and $A$ will be a unital $\kk$-algebra. 
We write $\ot=\ot_{\kk}$.
As usual,  $A^e=A\ot A^{op}$ denotes the enveloping algebra of $A$. 
We identify $A$-bimodules with (left) $A^e$-modules. 

\subsection*{Acknowledgements}
The first author was supported by NSF Postdoctoral Fellowship DMS-1503147. 
The second author was supported by RFBR according to the research project 17-01-00258, and by the President’s Program ``Support of Young Russian Scientists” (project number MK-2262.2019.1).
The third author was partially supported by NSF grants DMS-1401016 and DMS-1665286.  We thank Pedro Tamaroff for relevant references, and the referee for their thoughtful comments.

\section{$A_{\infty}$-coalgebras and $A_{\infty}$-coderivations}\label{sec:prelim}

In this section, we give many definitions that we will need in the paper.
In particular, we 
define $A_{\infty}$-coalgebras, $A_{\infty}$-coderivations,
and a cup product and a graded Lie bracket on $A_{\infty}$-coderivations.

\subsection{$A_{\infty}$-coalgebras and the cobar construction}\label{subsec:ccc}

Let $\mathscr{Q}$ be the category of dg $A$-bimodules, that is complexes of $A$-bimodules with morphisms the hom complex $\Hom_\mathscr{Q}(X,Y)$ 
of homogeneous maps between $X$ and $Y$ of arbitrary degrees.  
Then $\mathscr{Q}$ is a monoidal category under the product $\ox_A$ and unit $A$. 
%(One can more generally consider $\mathscr{Q}$ to be a monoidal category enhanced over the category of chain complexes over $\kk$.)

For an object $C$ in $\mathscr{Q}$, we let $T(C)$ denote the free augmented dg algebra in $\mathscr{Q}$, which is explicitly the sum of the tensor product complexes
\[
T(C)=\bigoplus_{n\geq 0} C^{\ox_A n}
\]
with the standard product given by concatenation and differential $d_{T(C)}=\sum_{n=1}^{\infty}d_{C^{\ot_A n}}$.  (Recall that $d_{X\ox_A Y}=d_X\ox 1+1\ox d_Y$, interpreted via the Koszul sign
convention when we apply the differential to elements.)  The dg algebra $T(C)$ is $A$-augmented with augmentation $\epsilon_T : T(C)\rightarrow A$ given by projection onto $T^0(C) = A$.

We consider the $\mfk{m}$-adic completion $\widehat{T}(C)$ at the augmentation ideal $\mfk{m}=\ker(\epsilon_T)$,
\[
\widehat{T}(C)=\varprojlim_n T(C)/(C^{\ot_A n})=\prod_{n\geq 0}C^{\ot_A n} ,
\]
a linear topological, augmented, dg algebra in $\msc{Q}$.

\begin{Def}\label{def:OG}
An {\em $A_\infty$-coalgebra} in $\mathscr{Q}$ is an object $C$ in $\mathscr{Q}$ with the additional information of a homogenous degree $1$, $A^e$-linear, continuous, algebra derivation $\delta:\widehat{T}(C)\to \widehat{T}(C)$ such that:
\begin{enumerate}
\item[(i)] $\delta$ preserves the augmentation, that is the composite 
$\epsilon_T\delta:\widehat{T}(C)\to \widehat{T}(C)\to A$ is~$0$;
\item[(ii)] $\delta$ solves the Maurer-Cartan equation in the endomorphism complex $\mathrm{End}_\mathscr{Q}(\widehat{T}(C))$: 
\[
d_{\mrm{End}}(\delta)+\delta\circ\delta=0.
\]
\end{enumerate}
\end{Def}

There is a notion of a ``weakly counital" $A_\infty$-coalgebra, and we will suppose that all of our $A_\infty$-coalgebras are weakly counital.  This definition is easier to give with a componentwise description of an $A_\infty$-structure on a complex $C$, and is given in Section~\ref{sect:components}.

The above definition of an $A_{\infty}$-coalgebra is constructed following the standard principles of $A_\infty$-(co)algebras outlined in, say, Lefevre-Hasegawa's thesis~\cite{Lefevre}.  (See also~\cite{Keller,KellerII,LPWZ}.)  Note that the continuous derivation $\delta$ is determined by a sequence of degree $1$ maps $\delta_n:C\to C^{\ox_A n}$ (see Section~\ref{sect:components} below for explicit conditions on $\delta_n$).  Compatibility with the augmentation implies that the term $\delta_{0}$ vanishes.  We make this (standard) assumption to avoid dealing with ``curved" structures.  Note also that if $\big((C,d_C),\delta\big)$ is an $A_\infty$-coalgebra, then $\big((C,0),\delta+d_C\big)$ and $\big((C,d_C+\delta_1),\delta-\delta_1\big)$ are $A_\infty$-coalgebras too.
\par

The familiar reader will recognize that an $A_\infty$-coalgebra can be described equivalently as a graded object in the category of $A$-bimodules with a sequence of maps $\Delta_n:C[1]\to C[1]^{\ox_A n}$ of degree $2-n$ satisfying certain compatibilities.  This definition is more natural from a classical perspective, where dg coalgebras are central.  Indeed, in the case that $\Delta_{n}$ vanishes for all $n>2$, $C=(C,\Delta_1,\Delta_2)$ will simply be a dg coalgebra in $\msc{Q}$. 
\par

Let us recall a standard construction, which in some sense appeared already in the work of Adams in the 1950's~\cite{Adams}.

\begin{Def}\label{def:cobar} 
For an $A_\infty$-coalgebra $C=(C,\delta)$ in $\mathscr{Q}$ we let $\Omega C$ denote the dg algebra in $\mathscr{Q}$ which is $\widehat{T}(C)$ as an augmented, topological, algebra, with differential $d_{\Omega C}=d_{\widehat{T}(C)}+\delta$.  We call $\Omega C$ the {\em cobar construction} of $C$.
\end{Def}

As before, $d_{\widehat{T}(C)}$ denotes the differential on the (completed) free dg algebra generated by $C$ in the above definition.  We work with the completed algebra $\widehat{T}(C)$ in order to allow the $\delta_n$ to be nonvanishing for arbitrarily large $n$.  Similarly, for morphisms, we would like to allow arbitrary $f_n$ to be nonvanishing.  
One otherwise might define the cobar construction via the tensor algebra $T(C)$ directly.
See Section~\ref{sec:conn-gr} on connected graded algebras where this is sometimes
sufficient for our connection between Hochschild cohomology and $A_{\infty}$-coderivations, that is our $A_{\infty}$-coalgebras there will be locally finite
in the following sense: 

\begin{Def}
An $A_\infty$-coalgebra $C$ in $\msc{Q}$ is called {\em locally finite} if for each $c\in C$ there exists an integer $N_c$ such that $\delta_N(c)=0$ for all $N>N_c$.
\end{Def}

In the case of a locally finite $A_\infty$-coalgebra $C$ the derivation $d_\Omega:\Omega C\to \Omega C$ is such that the restriction to the topological generators $C$ has image in the dense algebra $T(C)\subset \Omega C$.  Consequently, $d_\Omega$ restricts to a derivation on $T(C)$, and the cobar construction $\Omega C$ is the $\mfk{m}$-adic completion of the dg algebra $(T(C),\delta_\Omega|T(C))=(T(C),d_{T(C)}+\delta|_{T(C)})$.

\begin{Def}\label{def:discrete-cobar}
For a locally finite $A_\infty$-coalgebra $C$ in $\msc{Q}$ we define the {\em discrete cobar construction} as the dg algebra $\Omega^\mrm{disc}C=(T(C),d_{T(C)}+\delta|_{T(C)})$.
\end{Def}

Although in general it is appropriate to employ the completed algebra $\Omega C$, even when the $A_\infty$-structure on $C$ is locally finite, we will see in Section~\ref{sec:conn-gr} that under certain finiteness assumptions one can effectively employ the discrete algebra $\Omega^\mrm{disc}C$.

\subsection{DG Lie algebras of $A_{\infty}$-coderivations}\label{sect:coders}
Note that for any topological dg algebra $\Omega$ in $\mathscr{Q}$, the collection $\mathrm{Der}_\mathscr{Q}(\Omega)$ of continuous, graded, derivations of $\Omega$ in $\msc{Q}$ forms a subcomplex in the endomorphism complex $\End_\mathscr{Q}(\Omega)$.  Furthermore, since the (graded) commutator of derivations is another derivation, the commutator operation
\[
[f,g]=f\circ g-(-1)^{|f||g|}g\circ f
\]
of graded derivations $f,g:\Omega\to\Omega$ makes $\mathrm{Der}_\mathscr{Q}(\Omega)$ naturally a dg Lie algebra.  To be clear, the differential on $\mathrm{Der}_\mathscr{Q}(\Omega)$ is the adjoint operator
\[
d_{\mathrm{Der}}=[d_\Omega,-].
\]
This operator $d_{\mathrm{Der}}$ is a graded Lie algebra derivation as a consequence of the Jacobi identity. 
It satisfies the equation $d_{\mathrm{Der}}\circ d_{\mathrm{Der}}=0$ as a consequence of the identity $[f,[f,g]]=[f\circ f, g]$, which is valid for $f\in \End_\mathscr{Q}(\Omega)$ of odd degree and any $g\in \End_\mathscr{Q}(\Omega)$.  We apply these observations in particular to the case $\Omega=\Omega C$ for an $A_\infty$-coalgebra $C$ in $\mathscr{Q}$ to produce a dg Lie algebra
\[
\mathrm{Der}_\mathscr{Q}(\Omega C).
\]
By our previous discussions, $\mrm{Der}_\msc{Q}(\Omega C)$ is identified with $\Hom_{\mathscr{Q}}(C,\Omega C)$ as a vector space.
\par

We view our dg Lie algebra $\mrm{Der}_\msc{Q}(\Omega C)$ as analogous to the complex $\Coder(T(A))$ in~\cite{Stasheff}.  Continuing this analogy, we are interested in the cohomology of this complex.  The space of cocycles
$\mrm{Coder}^\infty_\mathscr{Q}(C)$
is referred to as the {\em $A_{\infty}$-coderivations}, and the
space of coboundaries $\mathrm{Inn}^\infty_\mathscr{Q}(C)$
is the space of {\em inner $A_\infty$-coderivations}.

As with any dg Lie algebra, the space of cocycles in $\mrm{Der}_\msc{Q}(\Omega C)$ forms a Lie subalgebra, the space of coboundaries forms a Lie ideal in the Lie subalgebra of cocycles, and the cohomology
\[
H^\bullet(\mrm{Der}_\msc{Q}(\Omega C))=\mrm{Coder}^\infty_\msc{Q}(C)/\mrm{Inn}^\infty_\msc{Q}(C)
\]
is a graded Lie algebra.
In Theorem~\ref{mainthm}, we will make a connection to Hochschild cohomology
of an algebra $A$ when $C$ is a bimodule resolution of $A$.  

\subsection{Componentwise definitions}
\label{sect:components}

In this section we give componentwise definitions of $A_{\infty}$-coalgebras
and $A_{\infty}$-coderivations that will be needed in calculations later.

Let $C=(C,\delta)$ be an $A_{\infty}$-coalgebra. 
The $A_\infty$-structure is specified by a collection of maps
\[
\delta_n:C\to C^{\ox_A n},\ \ \delta_n=p_n d_{\Omega C}|_{C},
\]
where $p_n$ is the linear projection $\Omega C\to C^{\ox_A n}$.
The equation $d_{\Omega C}^2=0$ can now be written in the form
\begin{equation}\label{eq:comp}
0=\sum_{r+s+t=N} (1^{\ox_A r}\ox_A \delta_s\ox_A 1^{\ox_A t})\delta_{N-s+1}
\end{equation}
for each $N\geq 0$.  We can also write this equation as
\begin{equation}\label{eqn:d-delta}
\partial(\delta_N)=-\sum_{\stackrel{1<s<N}{r+s+t=N}} (1^{\ox_A r}\ox_A \delta_s\ox_A 1^{\ox_A t})\delta_{N-s+1},
\end{equation}
where $\partial$ denotes the differential on the Hom complex $\Hom_\mathscr{Q}(C,C^{\ox_A N})$ for each $N$; the differential on the complex $C$ is $\delta_1$.

Thus, an $A_\infty$-coalgebra in $\mathscr{Q}$ is a graded $A$-bimodule $C$ equipped with $A^e$-linear operators $\delta_n:C\to C^{\ox_A n}$ of degree $1$, for $n\geq 1$, that solve the equation~\eqref{eq:comp} for all nonnegative $N$.

\begin{Def}
An $A_\infty$-coalgebra $C$ is called {\em weakly counital} if it comes equipped with a degree $1$ bimodule map $\mu:C\to A$ such that $(\mu\ot_A \mu)\delta_2=\mu$ and $(\mu^{\ot_An})\delta_n=0$ for all $n>2$.
\end{Def}

An $A_\infty$-coalgebra which is strictly counital (see e.g.~\cite{LPWZ}) is weakly counital.  All $A_\infty$-coalgebras here will be assumed to be weakly counital for convenience, although the assumption can be avoided.  We next give a componentwise definition of $A_{\infty}$-coderivation. 

\begin{Def}\label{def:coder_comp}
An {\em $A_\infty$-coderivation} $f$ of degree $l$ on $C$ is a 
sequence of maps 
\[
f_n:C\to C^{\ox_A n}
\]
of degree $l$,  for $n\geq 0$, that satisfy
\begin{equation}\label{eq:Ccomp}
\sum_{r+s+t=N} (1^{\ox_A r}\ox_A \delta_s\ox_A 1^{\ox_A t})f_{N-s+1}
=(-1)^l\sum_{r+s+t=N} (1^{\ox_A r}\ox_A f_{s}\ox_A 1^{\ox_A t})\delta_{N-s+1}
\end{equation}
for all $N$.
\end{Def}

We note the possible appearance of the scalar term $f_0:C\to A$ in the above 
notion of an $A_\infty$-coderivation.

We consider an $A_{\infty}$-coderivation to be an element of 
$\Hom_{\mathscr{Q}} (C, \Omega C)$, and identify it with the corresponding
function  in $\Der_{\mathscr{Q}}(\Omega C)$ as described earlier.
In this way we may compose two $A_{\infty}$-coderivations. 
The space of such maps is not closed under composition, but 
it is closed under a bracket, as we define next. 

For $f$ and $g$ as in Definition~\ref{def:coder_comp}, of respective degrees $l$ and $p$, we deduce a Lie bracket given componentwise as $[f,g]=([f,g]_N)_{N\geq 0}$ from our earlier description.  
This is the $A_\infty$-coderivation defined by 
\[
[f,g]_N= \sum_{r+s+t=N} (1^{\ox_A r}\ox_A f_{s}\ox_A 1^{\ox_A t})g_{N-s+1} -(-1)^{lp}\sum_{r+s+t=N}(1^{\ox_A r}\ox_A g_{s}\ox_A 1^{\ox_A t})f_{N-s+1}.
\]
The above operation $[ \ , \ ]$ extends immediately to an operation on the collection of tuples $(b_n)_{n\geq 0}$ of maps $b_n$ in $\Hom_{\mathscr{Q}}(C,C^{\ox_A n})$ such that all the maps $b_n$ have the same degree, and we understand from our earlier description that for any such $b$ the tuple $[\delta,b]$ is an $A_\infty$-coderivation. We call such $A_\infty$-coderivations {\em inner}, as before.

\subsection{The cup product on $A_\infty$-coderivations}\label{subsec:cupprod}

Our main Theorem~\ref{mainthm} below gives an isomorphism of Gerstenhaber algebras.
Towards this end, we define a cup product on $A_{\infty}$-coderivations. 
We first reiterate the definition of function composition in
$\Der_{\mathscr{Q}}(\Omega C)$, expressed componentwise: 

\begin{Def}\label{circ}
Let $(C,\delta)$ be an $A_{\infty}$-coalgebra. For $f\in \DerC_n$ and 
$g\in \DerC_m$, we define $f\circ g\in \Hom_{\mathscr{Q}}(C,\Omega C)_{n+m}$ by the equality
$$
(f\circ g)_i=\sum\limits_{r+s+t=i}
(1^{\ot_Ar}\ot_Af_s\ot_A1^{\ot_At})g_{r+t+1}.
$$
In terms of the elements of $\DerC$, this definition can be reformulated in the following way. If $f$ and $g$ are tuples of maps in $\Hom_{\mathscr{Q}}(C, \Omega C)$ that correspond to $\overline{f},\overline{g}\in \DerC$ respectively, then $f\circ g$ corresponds to the element $\overline{f\circ g}\in \DerC$ such that $(\overline{f\circ g}-\overline{f}\circ\overline{g})|_{C}=0$.
\end{Def}

\begin{Lemma}\label{preLie} For $f\in \DerC_n$, $g\in \DerC_m$ and $h\in \DerC_p$ one has
$$
(f\circ g)\circ h-f\circ (g \circ h)=(-1)^{mn}\big((g\circ f)\circ h-g\circ (f\circ h)\big).
$$
Thus, $(\DerC,\circ)$ is a graded pre-Lie algebra.
%Moreover, if $f$ has odd degree, then $f\circ (f\circ g)=(f\circ f)\circ g$.
\end{Lemma}
\begin{proof} Let as before $f$, $g$ and $h$ correspond to the elements $\overline{f}$, $\overline{g}$, and $\overline{h}\in \DerC$ respectively. Since $\overline{f}\circ \overline{g}-(-1)^{mn}\overline{g}\circ \overline{f}\in \DerC$, one sees that $(f\circ g-(-1)^{mn}g\circ f)\circ h$ corresponds to the element $H\in \DerC$ such that $H|_C=(\overline{f}\circ \overline{g}\circ \overline{h}-(-1)^{mn}\overline{g}\circ \overline{f}\circ \overline{h})|_C$. On the other hand, $f\circ (g \circ h)$ and $g\circ (f\circ h)$ correspond to the elements $F,G\in \DerC$ such that $(F-\overline{f}\circ \overline{g}\circ \overline{h})|_C=(G-\overline{g}\circ \overline{f}\circ \overline{h})|_C=0$. Thus, $H=F-(-1)^{mn}G$, and the required equality follows.
\end{proof}

We next define an $A_{\infty}$-algebra structure on $\DerC$ that is induced
by convolution of the $A_{\infty}$-coalgebra
structure and the algebra structure on $\Omega C$:

\begin{Def}\label{defn-cup}
Let $(C,\delta)$ be an $A_{\infty}$-coalgebra. 
Let $f_i\in \DerC_{p_i}$ for $1\le i\le l$, $l\ge 2$. Then we define $m_l(f_1,\dots,f_l)\in \DerC_{\sum\limits_{i=1}^lp_i+1}$ by the equality
$$
m_l(f_1,\dots,f_l)_n=(-1)^{\sum\limits_{i=1}^l(i-1)p_i}\sum\limits_{\sum\limits_{k=0}^{l}i_k+\sum\limits_{k=1}^{l}j_k=n}(1^{\ot_Ai_0}\ot_A(f_1)_{j_1}\ot_A\dots\ot_A(f_l)_{j_l}\ot_A1^{\ot_Ai_{l}})\delta_{\sum\limits_{k=0}^{l}i_k+l}.
$$
\end{Def}

Since we will be mainly interested in the case $l=2$, let us write down that, for $f\in \DerC_n$ and $g\in \DerC_p$, the element $f\smile g=m_2(f,g)\in \DerC_{n+m+1}$ is defined by the equality
$$
(f\smile g)_i=(-1)^p\sum\limits_{r+s+t+u+v=i}
(1^{\ot_Ar}\ot_Af_s\ot_A1^{\ot_At}\ot_Ag_u\ot_A1^{\ot_Av})\delta_{r+t+v+2}.
$$

\begin{Lemma}\label{derivate} For any $l$-tuple $(f_1,\dots,f_l)\in \DerC_{p_1}\times\dots\times \DerC_{p_l}$, one has
\begin{multline*}
[\delta,m_l(f_1,\dots,f_l)]=\sum\limits_{i=1}^l(-1)^{\sum\limits_{k=1}^{i-1}(p_k+1)}m_l(f_1,\dots,f_{i-1},[\delta,f_i],f_{i+1},\dots,f_l)\\
-(-1)^{\sum\limits_{k=1}^{l}p_k}\sum\limits_{j=2}^{l-1}\sum\limits_{i=0}^{l-j}(-1)^{i+j\sum\limits_{k=i+j+1}^{l}p_k}m_{l-j+1}\big(f_1,\dots,f_i,m_j(f_{i+1},\dots,f_{i+j}),f_{i+j+1},\dots,f_l\big).
\end{multline*}
\end{Lemma}
\begin{proof}Let us introduce 
\begin{multline*}
{\bf f}=m_l(f_1,\dots,f_l),\,\,{\bf f}^{\delta}_i=m_{l+1}(f_1,\dots,f_i,\delta,f_{i+1},\dots,f_l),\\{\bf f}^{\delta\circ}_i=m_l(f_1,\dots,f_{i-1},\delta\circ f_i,f_{i+1},\dots,f_l)\mbox{, and }{\bf f}^{\circ\delta}_i=m_l(f_1,\dots,f_{i-1},f_i\circ\delta,f_{i+1},\dots,f_l).
\end{multline*}

Let us look at the left hand side $[\delta,{\bf f}]=\delta\circ {\bf f}+(-1)^{\sum\limits_{k=1}^{l}p_k}{\bf f}\circ\delta$ of the required equality. Direct calculations show that
\begin{multline*}
\delta\circ {\bf f}=\sum\limits_{i=0}^l(-1)^{\sum\limits_{k=1}^{l}p_k+i}{\bf f}^{\delta}_i+\sum\limits_{i=1}^l(-1)^{\sum\limits_{k=1}^{i-1}(p_k+1)}{\bf f}^{\delta\circ}_i=\sum\limits_{i=0}^l(-1)^{\sum\limits_{k=1}^{l}p_k+i}{\bf f}^{\delta}_i\\
-\sum\limits_{i=1}^l(-1)^{\sum\limits_{k=1}^{i}(p_k+1)}{\bf f}^{\circ\delta}_i+\sum\limits_{i=1}^l(-1)^{\sum\limits_{k=1}^{i-1}(p_k+1)}m_l(f_1,\dots,f_{i-1},[\delta,f_i],f_{i+1},\dots,f_l).
\end{multline*}
Let us introduce the element $H\in\Der_\mathscr{Q}(\Omega C)_{\sum\limits_{i=1}^lp_i+2}$ by the equality
$$
H_n=(-1)^{\sum\limits_{k=1}^{l}kp_k}\sum\limits_{\sum\limits_{k=0}^{l}i_k+\sum\limits_{k=1}^{l}j_k=n}(1^{\ot_Ai_0}\ot_A(f_1)_{j_1}\ot_A\dots\ot_A(f_l)_{j_l}\ot_A1^{\ot_Ai_{l}})(\delta\circ\delta)_{\sum\limits_{k=0}^{l}i_k+l}.
$$
It follows from $\delta\circ \delta=0$ that $H=0$. On the other hand, it is not difficult to see that
\begin{multline*}
H=\sum\limits_{i=0}^l(-1)^{\sum\limits_{k=1}^{l}p_k+i}{\bf f}^{\delta}_i-\sum\limits_{i=1}^l(-1)^{\sum\limits_{k=1}^{i}(p_k+1)}{\bf f}^{\circ\delta}_i+(-1)^{\sum\limits_{k=1}^{l}p_k}{\bf f}\circ\delta\\
+(-1)^{\sum\limits_{k=1}^{l}p_k}\sum\limits_{j=2}^{l-1}\sum\limits_{i=0}^{l-j}(-1)^{i+j\sum\limits_{k=i+j+1}^{l}p_k}m_{l-j+1}\big(f_1,\dots,f_i,m_j(f_{i+1},\dots,f_{i+j}),f_{i+j+1},\dots,f_l\big).
\end{multline*}
Subtracting the obtained expression from the formula for $\delta\circ {\bf f}$ obtained above, we get the required equality.
\end{proof}

Lemma~\ref{derivate} says that $\mrm{Der}_\msc{Q}(\Omega C)$ is an $A_\infty$-algebra under the convolution operations $\{m_n\}_n$.  As with $A_\infty$-algebras more generally, the subspace of cocycles is not an $A_\infty$-subalgebra immediately, but the cocycles are closed under the (possibly non-associative) multiplication $m_2$.

\begin{coro}\label{ass} The operation $\smile \ = m_2$ induces a degree preserving operation on $\Coder^{\infty}_\mathscr{Q}(C)[1]$. Moreover, $\Inn^{\infty}_\mathscr{Q}(C)[1]$ is an ideal in $\Coder^{\infty}_\mathscr{Q}(C)[1]$ with respect to the operation $\smile$, and the algebra $\Big(\big(\Coder^{\infty}_\mathscr{Q}(C)/\Inn^{\infty}_\mathscr{Q}(C)\big)[1],\smile\Big)$ is associative.
\end{coro}
\begin{proof} Everything except the associativity follows directly from Lemma~\ref{derivate} with $l=2$. The stated associativity follows from Lemma~\ref{derivate} with $l=3$.
\end{proof}

\begin{Lemma}\label{gc} Let $(C,\delta)$ be an $A_{\infty}$-coalgebra. Then $$(-1)^p(f\circ (g\circ\delta)-(f\circ g)\circ\delta)=  f\smile g- (-1)^{(p+1)(n+1) }g\smile f$$ for any $f\in \End_\mathscr{Q}(C)_n$ and $g\in \End_\mathscr{Q}(C)_p$. In particular, if $f$ and $g$ are  $A_{\infty}$-coderivations, then
$f\smile g+ (-1)^{(p+1)(n+1) }g\smile f
=(-1)^{n}[\delta,f\circ g]\in \Inn^{\infty}_\mathscr{Q}(C)$.
\end{Lemma}
\begin{proof} The first equality easily follows from the definitions of the operations $\circ$ and $\smile$. The second equality follows from the first one and Lemma \ref{preLie}.
%For the second equation, note that by Lemma~\ref{preLie}, since $f$ and $g$ are coderivations
%and $| \Delta | =1$,
%$$
%\begin{aligned}
%  & f_a\circ (g_b\circ \Delta_c)\\  
%   & =  (-1)^{|g|+bc} f_a\circ (\Delta_c\circ g_b) \\
%    & =  (-1)^{|g|+bc} (f_a\circ \Delta_c) \circ g_b - (-1)^{|f|+|g|+ac+bc} (\Delta_c\circ f_a) \circ g_b 
%    + (-1)^{|f|+|g|+ac+bc} \Delta_c \circ (f_a\circ g_b) \\
%   &=  (-1)^{|f|+|g|+ac+bc} \Delta_c\circ (f_a\circ g_b) , 
%\end{aligned}
%$$
%so $\Delta_c \circ (f_a\circ g_b) = (-1)^{|f|+ac} f_a\circ (\Delta_c\circ g_b)$.
%From this and the hypothesis that $g$ is a coderivation, we find
%\begin{eqnarray*}
%  [f_a \circ g_b , \Delta_c] &=& (f_a\circ g_b)\circ \Delta_c - (-1)^{|f|+|g| + (a+b)c} \Delta_c
%   \circ (f_a\circ g_b) \\
%  &=& (f_a\circ g_b )\circ \Delta_c - (-1)^{|g|+bc} f_a \circ (\Delta_c\circ g_b) \\
%   &=& (f_a\circ g_b)\circ \Delta_c - f_a \circ (g_b\circ \Delta_c) ,
%\end{eqnarray*}
%and so $[f\circ g , \Delta ] = (f\circ g)\circ \Delta - f\circ (g\circ \Delta)$. 
\end{proof}

\begin{theorem} If $(C,\delta)$ is an $A_{\infty}$-coalgebra, then $\Inn^{\infty}_\mathscr{Q}(C)$ is an ideal in $\Coder^{\infty}_\mathscr{Q}(C)$ with respect to the operations $\smile$ and $[ \ , \ ]$. Moreover, $\Big(\big(\Coder^{\infty}_\mathscr{Q}(C)/\Inn^{\infty}_\mathscr{Q}(C)\big)[1],\smile,[ \ , \ ]\Big)$ is a Gerstenhaber algebra (in general, nonunital).
\end{theorem}

\begin{proof} It was shown above that $\smile$ and $[\ , \ ]$ have the required degrees, and $[ \ , \ ]$ satisfies the conditions of a  graded Lie algebra.

The required associativity and graded commutativity of $\smile$ follow from Lemma \ref{gc} and Corollary \ref{ass}. Thus, it remains to prove the Poisson identity.

For an $l$-tuple $(f_1,\dots,f_l)\in \DerC_{p_1}\times\dots\times \DerC_{p_l}$ and $h\in \DerC_q$, we define $m_l^h(f_1,\dots,f_l)\in \DerC_{\sum\limits_{i=1}^lp_i+q}$ by the equality
$$
m_l^h(f_1,\dots,f_l)_n=\sum\limits_{\sum\limits_{k=0}^{l}i_k+\sum\limits_{k=1}^{l}j_k=n}(1^{\ot_Ai_0}\ot_A(f_1)_{j_1}\ot_A\dots\ot_A(f_l)_{j_l}\ot_A1^{\ot_Ai_{l}})h_{\sum\limits_{k=0}^{l}i_k+l}.
$$
Let us take $f\in \DerC_n$, $g\in \DerC_r$ and $h\in \DerC_q$.
Direct calculations show that
\begin{multline*}
[f\smile g,h]=(f\smile g)\circ h+(-1)^{(r+n+1)(q+1)}m_3(h,f,g)-(-1)^{(r+n+1)q}(h\circ f)\smile g\\
-(-1)^{r(q+1)}m_3(f,h,g)-(-1)^{rq}f\smile(h\circ g)-(-1)^{q}m_3(f,g,h)=(f\smile g)\circ h\\
+(-1)^{(r+n+1)(q+1)}m_3(h,f,g)-(-1)^{r(q+1)}m_3(f,h,g)-(-1)^{q}m_3(f,g,h)\\
-(-1)^{(r+1)q}(f\circ h)\smile g-f\smile(g\circ h)+(-1)^{(r+1)q}[f,h]\smile g+f\smile[g, h],
\end{multline*}
\begin{multline*}
[\delta,m^h_2(f,g)]=m^h_3(\delta,f,g)+m^h_2(\delta\circ f,g)+(-1)^nm^h_3(f,\delta,g)+(-1)^nm^h_2(f,\delta\circ g)\\
+(-1)^{r+n}m^h_3(f,g,\delta)-(-1)^{r+n+q}m^h_2(f,g)\circ \delta,
\end{multline*}
\begin{multline*}
m_2^{\delta\circ h}(f,g)=(-1)^{r+n}m^h_3(\delta,f,g)+(-1)^rm^h_2(f\circ \delta,g)+(-1)^{r}m^h_3(f,\delta,g)\\
+m^h_2(f,g\circ \delta)+m^h_3(f,g,\delta)+(-1)^r(f\smile g)\circ h,
\end{multline*}
\begin{multline*}
m_2^{h\circ \delta}(f,g)=(-1)^{(r+n)q+n}m_3(h,f,g)+(-1)^{r(q+1)}(f\circ h)\smile g\\
+(-1)^{(r+1)q}m_3(f,h,g)+(-1)^{r+q}f\smile (g\circ h)+(-1)^rm_3(f,g,h)+m_2^h(f, g)\circ \delta.
\end{multline*}
Hence,
\begin{multline*}
[f\smile g,h]-f\smile [g,h]-(-1)^{(r+1)l}[f,h]\smile g=(-1)^rm_2^{[\delta, h]}(f,g)\\
+(-1)^nm^h_2([\delta,f],g)+m^h_2(f,[\delta, g])-(-1)^{n}[\delta,m^h_2(f,g)].
\end{multline*}
In particular, if $f,g,h \in \Coder^{\infty}_{\mathscr{Q}}(C)$, then
$$[f\smile g,h]-f\smile [g,h]-(-1)^{(r+1)l}[f,h]\smile g=(-1)^{n+1}[\delta,m^h_2(f,g)]\in\Inn^{\infty}_\mathscr{Q}(C).$$
\end{proof}

\begin{rema}
One should note the similarity between our definition of the higher multiplications $m_l$ on $\mrm{Der}_\msc{Q}(\Omega C)$ and the brace operations which appear on usual Hochschild cochains~\cite[Section 3]{vg95}.  Indeed, the $m_l$ are essentially given by certain dualized brace operations.
\end{rema}

\section{$A_{\infty}$-coalgebra structure on a projective resolution}\label{sec:proj-res}

In this section, we show how to put an $A_{\infty}$-coalgebra structure on a projective $A$-bimodule resolution of the $\kk$-algebra $A$.
We will consider resolutions of $A$ that are shifted:
Let $P = P'[-1]$ where $P'$ is a projective $A$-bimodule resolution of $A$, so that $P_1 = P_0'$ and there is an $A$-bimodule chain map $\mu_P:P\rightarrow A$ of degree $-1$ that induces an isomorphism ${\rm H}_*(P)\cong A[-1]$. 
We will call this complex $P$  a {\it shifted projective bimodule resolution} of $A$. 

Let $X$ and $Y$ be $A$-bimodule complexes and $f:X\rightarrow Y$  a homogeneous $A$-bimodule homomorphism (equivalently an $A^e$-module homomorphism where
$A^e=A\ot A^{op}$). 
Denote the differential on $\Hom_{\mathscr{Q}}(X,Y)$ by $\partial$,
that is 
$$
     \partial(f) = d_Yf-(-1)^{|f|}fd_X.
$$
We will write also $f\sim g$ if $f$ is homotopic to $g$.

\subsection{Existence of an $A_{\infty}$-coalgebra structure}
We show that any projective resolution of the $A^e$-module $A$
has an $A_{\infty}$-coalgebra structure of a particular type. 

\begin{theorem}\label{Acostruc}
Let $(P,\mu_P)$ be a shifted bimodule resolution of $A$. Then $P$ admits a weakly counital
$A_{\infty}$-coalgebra structure $\delta$ with $\delta_1=d_{P}$.
\end{theorem} 

\begin{proof} There exists such a $\delta_2$ by uniqueness of projective resolutions, and the fact that $P\ot_A P$ is a (shift of a) resolution of $A$.  Consider the exact sequence of complexes
\[
0\to K_n\to \Hom_\msc{Q}(P,P^{\ot_A n})\stackrel{(\mu^{\ot_An})_\ast}{\relbar\joinrel\relbar\joinrel\relbar\joinrel\longrightarrow} \Hom_\msc{Q}(P,A)[-n]\to 0,
\]
where $K_n$ is simply the kernel of $(\mu^{\ot_A n})_*$.  Note that since $(\mu_P^{\ot_An})_\ast$ is a quasi-isomorphism, the complex $K_n$ is acyclic.  One sees this by considering the long exact sequence on cohomology.

Suppose that $n>2$ and we have already constructed $\delta_i$ for $i<n$ in such a way that $(\delta\circ\delta)_i=0$ for all $i<n$, and $\delta_i\in K_i$ for $i>2$.  Let us prove that there exists $\delta_n$ such that $(\delta\circ\delta)_n=0$. 
The required equality is
\begin{equation}\label{eq:600}
0=\sum\limits_{i=1}^n\delta_i\circ\delta_{n-i+1}=\sum\limits_{i=2}^{n-1}\delta_i\circ\delta_{n-i+1}+\partial(\delta_n).
\end{equation}
One applies $\mu_P^{\ot_A n}$ to see that $\sum_{i=2}^{n-1}\delta_i\circ\delta_{n-i+1}$ is in $K_n$.  Since $K_n$ is acyclic, the desired $\delta_n$ exists if and only if this sum is a cocycle.

Applying $\partial $ to $\sum\limits_{i=2}^{n-1}\delta_i\circ\delta_{n-i+1}$, we obtain
%\begin{multline*}
\begin{eqnarray*}
\partial\left(\sum\limits_{i=2}^{n-1}\delta_i\circ\delta_{n-i+1}\right)&=&\sum\limits_{i=2}^{n-1}\partial(\delta_i)\circ\delta_{n-i+1}-\sum\limits_{i=2}^{n-1}\delta_i\circ\partial(\delta_{n-i+1})\\
&=& - \sum\limits_{i=2}^{n-1}\left(\sum\limits_{j=2}^{i-1}\delta_j\circ\delta_{i-j+1}\right)\circ\delta_{n-i+1}
+\sum\limits_{i=2}^{n-1}\delta_i\circ\left(\sum\limits_{j=2}^{n-i}\delta_j\circ\delta_{n-i-j+2}\right)\\ &=&\sum\limits_{\scriptsize\begin{array}{c}i+j+k=n+2,\\i,j,k\ge 2\end{array}}\big( - (\delta_i\circ\delta_{j})\circ\delta_{k}+\delta_i\circ(\delta_{j}\circ\delta_{k})\big)=0.
\end{eqnarray*}
%\end{multline*}
Hence, $\sum\limits_{i=2}^{n-1}\delta_i\circ\delta_{n-i+1}$ is a chain map from $P$ to $P^{\ot_An}$ of degree $2$, and we can find the desired solution $\delta_n$ to equation~\eqref{eq:600}.
\end{proof}

Of course, the theorem tells us that any shifted resolution $\mu_P:P\to A$ can be endowed with an $A_\infty$-coalgebra structure such that $\mu_P$ provides a weak counit on $P$.  Throughout this work, when we speak of an $A_\infty$-coalgebra structure on a shifted resolution, we mean one such that the structure map $\mu_P:P\to A$ provides a weak counit.

\subsection{Examples of $A_{\infty}$-coalgebra structures}

We next give some examples, beginning with bar resolutions  having
coassociative  coalgebra
structures (i.e., $A_{\infty}$-coalgebra structures $\delta$ with $\delta_m=0$
for all $m\geq 3$) in Example~\ref{ex:bar}.
Then we give in Example~\ref{ex:xn} an  
$A_{\infty}$-coalgebra structure on a projective bimodule resolution that 
includes some nonzero maps $\delta_m$ for $m\geq 3$.

\begin{example}\label{ex:bar} 
Let $P$ be the shifted bar resolution of the $\kk$-algebra $A$:
\[
   P: \hspace{2cm} \cdots \stackrel{d}{\longrightarrow} A^{\ot 4}
  \stackrel{d}{\longrightarrow} A^{\ot 3} 
  \stackrel{d}{\longrightarrow} A\ot A 
\]
where $\ot = \ot_{\kk}$ and 
$d (a_0\ot \cdots \ot a_{n+1}) = \sum_{i=0}^{n} (-1)^i a_0\ot\cdots 
\ot a_ia_{i+1}\ot\cdots \ot a_{n+1}$ for all $a_0,\ldots, a_{n+1}\in A$, 
with augmentation $\mu_P : A\ot A\rightarrow A$ given by the multiplication
on $A$. 
That is, $P_1 = A\ot A$, $P_0=A^{\ot 3}$, $P_{-1}=A^{\ot 4}$, and so on. 
The standard diagonal map $\Delta_2$ for $P$ may be modified
by including some signs to obtain
$\delta_2: P\ot _A P \rightarrow P$ 
defined by
\[
   \delta_2 (a_0\ot\cdots\ot a_{n+1})  =
    \sum_{i=0}^{n} (-1)^i(a_0\ot\cdots\ot a_i\ot 1) \ot_A (1\ot a_{i+1}\ot
   \cdots\ot a_{n+1})
\]
for all $a_0,\ldots, a_{n+1}\in A$.
We may take $\delta_i = 0$ for all $i\geq 3$.
\end{example}

\begin{example}\label{ex:xn}
Let $A= \kk [x]/(x^n)$, $n>2$. (The case $n=2$ is Koszul; 
see Example~\ref{ex:x2} and Section~\ref{sec:Koszul}.)   
Consider
\[
   P : \hspace{2cm} \cdots\stackrel{\cdot v}{\longrightarrow} A\ot A 
  \stackrel{\cdot u}{\longrightarrow} A\ot A
  \stackrel{\cdot v}{\longrightarrow} A\ot A
  \stackrel{\cdot u}{\longrightarrow} A\ot A  ,
\]
where $\ot = \ot_{\kk}$,
$ u = x\ot 1 - 1\ot x$ and $v = x^{n-1}\ot 1 + x^{n-2}\ot x + \cdots
+ 1\ot x^{n-1}$, with augmentation
$\mu_P: A\ot A\rightarrow A$ given by the multiplication on $A$. 
For each $i$, let $e_i$ denote the element $1\ot 1$ in $A\ot A$ in 
degree~$1-i$.
By convention, we set $e_i=0$ whenever $i<0$.
A diagonal map $\Delta_2$ is given in~\cite[Section~5]{Negron-Witherspoon}
in case $ n =p $ and char$(\kk)=p$.
We modify it 
to define $\delta_2$ here for all $n\geq 3$, independent of characteristic:
\begin{eqnarray*}
  \delta_2(e_{2i}) & = & \sum_{j+l=i} e_{2j}\ot_A e_{2l} 
      -\sum_{\substack{j+l=i\\ a+b+c=n-2}} x^a e_{2j+1}\ot_A x^be_{2l-1}x^c ,\\
  \delta_2(e_{2i+1}) &=& \sum_{j+l=i} e_{2j}\ot_A e_{2l+1} -
   \sum_{j+l=i} e_{2j+1}\ot_A e_{2l} .
\end{eqnarray*}
For $3\le m\le n$, let
\[
\begin{aligned}
  & \delta_m (e_{2i}) \\
  & = (-1)^{m+1} \!\!\! 
   \sum_{\substack{j_1+\cdots +j_m = i-1\\ a_1+\cdots + a_{m+1} = n-m}}
     \!\!  x^{a_1}e_{2j_1+1} \ot_A x^{a_2}e_{2j_2+1}\ot_A\cdots\ot_A 
       x^{a_{m-1}}e_{2j_{m-1}+1}\ot_A  x^{a_m} e_{2j_m+1}x^{a_{m+1}} ,
\end{aligned}
\]
$\delta_m (e_{2i+1})=0$ for all $i$,
and $\delta_m = 0$ for all $m > n$. 
This is an $A_{\infty}$-coalgebra structure on $P$, as
may be checked via some lengthy calculations.
\end{example}

\section{$A_{\infty}$-coderivations on a projective resolution}\label{sec:main}

In this section, we work with an arbitrary shifted bimodule 
resolution $P$ (i.e.~$P = P'[-1]$ for some projective $A$-bimodule
resolution $P'$ of $A$).
We will show that Hochschild cocycles on $P$ give rise to
$A_{\infty}$-coderivations and 
Hochschild coboundaries correspond to inner $A_{\infty}$-coderivations.
We will give some examples and prove our main theorem. 

\subsection{Main Theorem}
We will prove the following 
isomorphism of Gerstenhaber algebras.
The notation $\Coder^{\infty}_{\mathscr{Q}}(P)$ and 
$\Inn^{\infty}_\mathscr{Q}(P)$ were introduced in Section~\ref{subsec:ccc},
and the associative and Lie products on the quotient space 
$\Coder^{\infty}_\mathscr{Q}(P)/\Inn^{\infty}_\mathscr{Q}(P)$ were
introduced in Sections~\ref{subsec:ccc} and~\ref{subsec:cupprod}. 

\begin{theorem}\label{mainthm}
Let $P$ be a shifted projective bimodule resolution of $A$ with $A_{\infty}$-coalgebra
structure $\delta$. 
There is an isomorphism of Gerstenhaber algebras 
\[
    \HH^*(A)\cong
   \big(\Coder^{\infty}_\mathscr{Q}(P)/\Inn^{\infty}_\mathscr{Q}(P)\big)[1] .
\]
\end{theorem}

The proof is deferred to Section~\ref{subsec:proof} and consists
of several steps, relying on some lemmas and preliminary results.  Our method of proof is very direct.  Indeed, one can view Theorem~\ref{mainthm} as a statement about the existence of solutions to an infinite sequence of equations.  We show directly that, in fact, one can produce solutions to all of the given equations.
\par
 
In Lemma~\ref{InnCoder} we explicitly identify the inner 
$A_{\infty}$-coderivations among all $A_{\infty}$-coderivations. 
We then give in Theorem~\ref{lift_exist} a construction of an \
$A_{\infty}$-coderivation corresponding to a Hochschild cocycle, 
and show the connection to the homotopy lifting method of~\cite{Volkov}. 
For this purpose, in Sections~\ref{subsec:inner}--\ref{subsec:proof}, 
we fix a shifted projective bimodule resolution $P$ 
of $A$ and a weakly counital $A_{\infty}$-coalgebra structure on $P$.
By Theorem~\ref{Acostruc}, such a structure always exists.

\subsection{A remark on methods}

Let us consider an $A$-bimodule resolution $\mu_{P'}:P'\to A$.  For any positive $n$, the tensor power $\mu_{P'}^{\ot_An}:(P')^{\ot_A n}\to A$ is also a resolution of $A$, and there are quasi-isomorphisms 
\[
1^{\ot_A m}\ot_A \mu_{P'}^{\ot_A(n-m)}:P^{\ot_A n}\to P^{\ot_A m}
\]
for each $m\le n$.  In the shifted case, where $P=P'[-1]$ and $\mu_P=\mu_{P'}$ after forgetting the grading, we still have quasi-isomorphisms $1^{\ot_A m}\ot_A \mu_P^{\ot_A(n-m)}:P^{\ot_A n}\to P^{\ot_A m}$.
\par

Since each positive power $P^{\ot_A w}$ is bounded above and projective, the induced maps
\[
(1^{\ot_A m}\ot\mu^{\ot_A (n-m)})_\ast:\Hom_\msc{Q}(P^{\ot_A w},P^{\ot_A n})\to\Hom_\msc{Q}(P^{\ot_A w},P^{\ot_A m})
\]
are all quasi-isomorphisms.  In particular, each map $(\mu^{\ot_A n})_\ast$ to $\Hom_\msc{Q}(P^{\ot_A w},A)$ is a quasi-isomorphism.  So, if we would like to know whether a given cocycle $f$ in a hom complex $\Hom_\msc{Q}(P^{\ot_A w},P^{\ot_A n})$ is a coboundary, it suffices to check whether $f$ is a coboundary after composing with some power of $\mu_P$.
\par

We will make copious use of this fact throughout the section in order to prove that a given map $f$ admits a bounding morphism $\phi$, i.e.\ $\phi$ with $\partial(\phi)=f$.  Most often, we will check after composing with the highest possible power $\mu^{\ot_A n}$.  For example, we have the following useful lemma.

\begin{lemma}\label{lem:940}
Let $P$ be a shifted bimodule resolution of $A$ which is endowed with a weakly counital $A_\infty$-coalgebra structure, as in Theorem~\ref{Acostruc}.  Then the two maps $(1\ot \mu_P)\delta_2$ and $-(\mu_P\ot 1)\delta_2:P\to P$ are homotopic to the identity.
\end{lemma}

\begin{proof}
Recall that $\delta_2$ satisfies the equation $(\mu_P\ot \mu_P)\delta_2=\mu_P$.  Hence the endomorphisms $(1\ot\mu_P)\delta_2$, $-(\mu_P\ot 1)\delta_2$, and the identity map $1_P$ all agree after applying $\mu_P:P\to A$.  It follows that all of the given maps are homotopic in $\End_\msc{Q}(P)$.
\end{proof}

\subsection{Inner coderivations}\label{subsec:inner}

Recall that $P=(P,\mu_P)$ is a shifted bimodule resolution with a corresponding weakly unital $A_\infty$-coalgebra structure.  We first give a characterization of inner $A_{\infty}$-coderivations. 

\begin{Lemma}\label{InnCoder}
Suppose that $\alpha\in \Coder^{\infty}_\mathscr{Q}(P)$. Then $\alpha\in\Inn^{\infty}_\mathscr{Q}(P)$ if, and only if, $\alpha_0\sim 0$.
\end{Lemma}

\begin{proof} 
If $\alpha\in\Inn^{\infty}_\mathscr{Q}(P)$, then by definition $\alpha=[\delta,\beta]$ for some $\beta\in\Der_\mathscr{Q}(\Omega C)$, and hence $\alpha_0=\partial(\beta_0)\sim 0$.

Suppose now that $\alpha_0\sim 0$. Then $\alpha_0= \partial (\beta_0)$ for
some $\beta_0$.
Define $\beta^{(0)}\in \mrm{Der}_\msc{Q}(\Omega C)$ by $\beta^{(0)}_0=\beta_0$ and $\beta^{(0)}_{>0}=0$.  Then $(\alpha-[\delta,\beta^{(0)}])_0=0$.

Let $\alpha^{(1)} = \alpha - [\delta,\beta^{(0)}]$ so that $\alpha_0^{(1)} = 0$.
We may assume that $\alpha$ is homogeneous of some degree~$n$.
We will prove by induction on $m$ that if $\alpha^{(m)}\in \Coder^{\infty}_\mathscr{Q}(P)_n$ is such that $\alpha^{(m)}_k=0$ for $0\le k\le m-1$, then there is some $\beta^{(m)}\in\Der_\mathscr{Q}(\Omega C)$ such that $\beta^{(m)}_i=0$ for $i\not\in\{m-1,m\}$ and $(\alpha^{(m)}-[\delta,\beta^{(m)}])_k=0$ for $0\le k\le m$. 
The assertion of the lemma will follow as then 
$\alpha = [\delta,  \sum_{i=0}^\infty \beta^{(i)}]$.
Let us prove this fact.

For simplicity, we will drop superscripts, writing $\alpha$ in place
of $\alpha^{(m)}$:
Assume that $\alpha\in \Coder^{\infty}_{{\mathscr{Q}}}(P)_n$ is such that $\alpha_k =0$ for $ 0\le k \le m-1$.
Applying Definition~\ref{def:coder_comp}, we thus have 
$\partial(\alpha_m)=0$ and $$\partial(\alpha_{m+1})+\sum\limits_{i=0}^{m-1}(1^{\ot_Ai}\ot_A\delta_2\ot_A1^{\ot_A(m-i-1)})\alpha_m-(-1)^n(\alpha_m\ot_A1+1\ot_A\alpha_m)\delta_2=0.$$
Suppose that $2\nmid m$.  Compose this equality with $\mu_P^{\ot_A(m+1)}$. 
Since $\mu_P^{\ot _A (m+1)} \partial (\alpha_{m+1})$ is a coboundary,
we get 
\begin{eqnarray*}
  0 & \sim & \mu_P^{\ot _A (m+1)} \sum_{i=0}^{m-1} (1^{\ot _A i}\ot_A \delta_2
    \ot_A 1^{\ot _A (m-i-1)} )\alpha_m 
     - (-1)^{n} \mu_P^{\ot _A (m+1)} (\alpha_m \ot_A 1 + 1\ot_A \alpha_m)
   \delta_2 \\
  &=& \sum_{i=0}^{m-1} (-1)^{ m-i-1} (\mu_P^{\ot_A i}\ot _A \mu_P\ot_A
   \mu_P^{\ot_A (m-i-1)} ) \alpha_m 
   - \mu_P^{\ot_A m} \alpha_m (1\ot_A \mu_P
   - \mu_P\ot_A 1)\delta_2 \\
 & = &  
     \mu_P^{\ot_Am}\alpha_m-\mu_P^{\ot_Am}\alpha_m(1\ot_A\mu_P-\mu_P\ot_A1)\delta_2.
\end{eqnarray*}
Now $(1\ot_A\mu_P)\delta_2$ and $ - (\mu_P\ot_A 1)\delta_2$ are both homotopic
to the identity map.
It follows that 
$(1\ot_A\mu_P-\mu_P\ot_A1)\delta_2\sim 2$.
Thus $\mu_P^{\ot_Am}\alpha_m\sim 0$, and consequently $\alpha_m\sim 0$. 
Now we can define $\beta_i=0$ for $i\not=m$ and choose $\beta_m$ such that $\alpha_m=\partial(\beta_m)$.

Suppose now that $2\mid m$.  Here we still have $\partial(\alpha_m)=0$. Define $\beta\in \mrm{Der}_\msc{Q}(\Omega C)$ by taking $\beta_i=0$ for $i\neq m-1$ and choose $\beta_{m-1}$ in such a way that $\partial(\beta_{m-1})=0$ and $\mu_P^{\ot_A(m-1)}\beta_{m-1}+\mu_P^{\ot_Am}\alpha_m=0$.  We can choose $\beta_{m-1}=-(1^{\ot_A (m-1)}\ot_A \mu_P)\alpha_m$, for example.  In this case $(\alpha-[\delta,\beta])_k=0$ for $0\le k\le m-1$.
Let us consider
\[
H=\alpha_m-\sum\limits_{i=0}^{m-2}(1^{\ot_Ai}\ot_A\delta_2\ot_A1^{\ot_A(m-i-2)})\beta_{m-1}-(-1)^n(\beta_{m-1}\ot_A1+1\ot_A\beta_{m-1})\delta_2.
\]
Applying $\partial$ to the above equation, since 
$\partial (\alpha_m)=0$, $\partial(\beta_{m-1})=0$, and $\partial (\delta_2)=0$,
we find that $\partial(H)=0$.
Applying $\mu_P^{\ot_A m}$ to the above definition of $H$, since
$(\mu_P\ot_A\mu_P)\delta_2 = \mu_P$ by hypothesis, we have 
$$
\mu_P^{\ot_Am}H\sim \mu_P^{\ot_Am}\alpha_m-\mu_P^{\ot_A(m-1)}\beta_{m-1}+\mu_P^{\ot_A(m-1)}\beta_{m-1}(1\ot_A\mu_P-\mu_P\ot_A1)\delta_2 .
$$
Since $\mu_P^{\ot_A m}\alpha_m = - \mu_P^{\ot_A (m-1)}\beta_{m-1}$ and 
$(1\ot_A \mu_P - \mu_P\ot_A 1)\delta_2 \sim 2$, it follows
that $\mu_P^{\ot_A m } H \sim 0$.
Thus we can choose $\beta_m$ in such a way that $(\alpha-[\delta,\beta'])_m=0$, for $\beta'$ with $\beta'_i=0$ whenever $i\notin \{m-1,m\}$, $\beta'_{m-1}=\beta_{m-1}$ and $\beta'_m=\beta_m$.
\end{proof}

\subsection{Lifting cocycles to $A_{\infty}$-coderivations}

In order to prove our main Theorem~\ref{mainthm}, we will need to 
prove that any Hochschild cocycle on $P$ can be lifted to an $A_{\infty}$-coderivation of $(P,\delta)$. For this purpose we need some additional 
technical lemmas and notation.

First, let us introduce for each $t\ge 1$ and integer $r$ a map $\psi_t^r\in \mrm{End}_\msc{Q}(P)$ of degree $1-t$.  We take
\[
\psi_t^r=\sum\limits_{u=0}^t(-1)^{ru}(\mu_P^{\ot_Au}\ot_A1\ot_A\mu_P^{\ot_A(t-u)})\delta_{t+1}
\]
for $t>1$, or $t=1$ and $r$ even, and otherwise, for all $l$,
\[
\psi_1^{2l+1}=(1\ot_A\mu_P-\mu_P\ot_A1)\delta_{2}-1 . 
\]
These maps will be used in Lemma~\ref{main_tec}, an essential step in
the proof of Theorem~\ref{lift_exist} that states a Hochschild cocycle
can be lifted to an $A_{\infty}$-coderivation.
We first prove some properties of the maps $\psi^r_t$.

\begin{Lemma}\label{psider}
The above functions $\psi_t^r$ satisfy the equations 
\[
\partial(\psi_t^r)+\sum\limits_{s=1}^{t-1}(-1)^{s}\psi_s^r\psi_{t-s}^{r-s}=0.
\]
\end{Lemma}

\begin{proof} Let us introduce $\mu_k^l=\mu_P^{\ot_Ak}\ot_A1\ot_A\mu_P^{\ot_A(l-k)}$. Then, applying~(\ref{eqn:d-delta}), we have 
\begin{equation}\label{partpsi}
\partial(\psi_t^r)=\sum\limits_{u=0}^t\sum\limits_{s=1}^{t-1}\sum\limits_{v=0}^{t-s}(-1)^{ru+t+1}\mu_u^t\big(1^{\ot_Av}\ot_A\delta_{s+1}\ot_A1^{\ot_A(t-s-v)}\big)\delta_{t-s+1}.
\end{equation}
Note that
\begin{multline*}
\mu_u^t\big(1^{\ot_Av}\ot_A\delta_{s+1}\ot_A1^{\ot_A(t-s-v)}\big)\delta_{t-s+1}\\
=\begin{cases}(-1)^{vs+t-s}\big(\mu_{u-v}^{s}\delta_{s+1}\big)\big(\mu_v^{t-s}\delta_{t-s+1}\big),&\mbox{if $v\le u\le v+s$},\\
(-1)^{t-v}\mu_{u-1}^{t-1}\delta_t,&\mbox{if $s=1$ and $0\le v\le u-2$},\\
(-1)^{t-v-1}\mu_{u}^{t-1}\delta_t,&\mbox{if $s=1$ and $u+1\le v\le t-1$},\\
0,&\mbox{otherwise.}
\end{cases}
\end{multline*}
Substituting the obtained values into \eqref{partpsi} and considering separately the cases $2\mid t$ and $2\nmid t$, we find that $\partial(\psi_t^r)=\sum\limits_{s=1}^{t-1}(-1)^{s+1}\psi_s^r\psi_{t-s}^{r-s}$.
\end{proof}

We next use the maps $\psi^r_t$ in the definition of additional maps
$\phi_t$ that will be used in the proof of Lemma~\ref{main_tec}.

\begin{Lemma}\label{phi} 
There exist maps $\phi_t:P\rightarrow P$ ($t\ge 1$) such that $|\phi_t|=-t$ and
$$
\partial(\phi_t)=\psi_t^{t-1}-\sum\limits_{s=1}^{t-1}\psi_s^{t-1}\phi_{t-s}
$$
for all $t\ge 1$.
\end{Lemma}

\begin{proof} For $t=0$ we simply need to show that $\psi_1^0$ is a coboundary.  
Since $\delta_2$ is a cocycle and $\psi_1^0=(1\ot_A\mu_P+\mu_P\ot_A1)\delta_2$, we have $\partial(\psi_1^0)=0$.  Since $\mu_P\psi_1^0=0$, $\psi_1^0$ is indeed a coboundary, and there is a map $\phi_1$ satisfying the required conditions. Let us now proceed by induction on $t$. Suppose that $t\ge 2$ and we have constructed $\phi_i$ for $1\le i\le t-1$ satisfying the required conditions. Let 
\[
   \Phi=\psi_t^{t-1}-\sum\limits_{s=1}^{t-1}\psi_s^{t-1}\phi_{t-s} .
\]
Since $\Phi$ has degree $1-t<0$, it is enough to verify that it is a chain map. Using Lemma~\ref{psider}, we get
\begin{multline*}
\partial(\Phi)=\sum\limits_{u=1}^{t-1}(-1)^{u+1}\psi_u^{t-1}\psi_{t-u}^{t-u-1}-\sum\limits_{s=1}^{t-1}\sum\limits_{u=1}^{s-1}(-1)^{u+1}\psi_u^{t-1}\psi_{s-u}^{t-u-1}\phi_{t-s}+\sum\limits_{s=1}^{t-1}(-1)^s\psi_s^{t-1}\partial(\phi_{t-s})\\
=\sum\limits_{s=1}^{t-1}(-1)^{s+1}\psi_s^{t-1}\big(\psi_{t-s}^{t-s-1}-\sum\limits_{r=1}^{t-s-1}\psi_r^{t-s-1}\phi_{t-s-r}-\partial(\phi_{t-s})\big)=0.
\end{multline*}
\end{proof}

\begin{Lemma}\label{main_tec} Let $i$ be any nonnegative integer. Suppose that $\alpha\in\DerP_n$ is such that $[\delta,\alpha]_j=0$ for $j\le 2i$. Then $\mu_P^{\ot_A(2i+1)}[\delta,\alpha]_{2i+1}\sim 0$.
\end{Lemma}

\begin{proof}
We apply the weakly counital property to find 
\[
\mu_P^{\ot_A(2i+1)}[\delta,\alpha]_{2i+1}=\mu_P^{\ot_A(2i+1)}\partial(\alpha_{2i+1})-(-1)^n\sum\limits_{j=0}^{2i}(-1)^{(2i+1-j)n}\mu_P^{\ot_Aj}\alpha_j\psi_{2i+1-j}^{2i-j}
\]
\begin{equation}\label{eq:1088}
\sim -\sum\limits_{j=0}^{2i}(-1)^{(2i-j)n}\mu_P^{\ot_Aj}\alpha_j\psi_{2i+1-j}^{2i-j}.
\end{equation}
Let us introduce $A_t=\sum\limits_{j=0}^{2i-t}(-1)^{(2i-j)n}\mu_P^{\ot_Aj}\alpha_j\left(\psi_{2i+1-j}^{2i-j}-\sum\limits_{s=2i-t-j+1}^{2i-j}\psi_s^{2i-j}\phi_{2i+1-j-s}\right)$, where the functions $\phi_u$ are as in Lemma~\ref{phi}.  Note that $-A_0$ is the final term in the above sequence~\eqref{eq:1088}, and $A_k=0$ for $k>2i$.  By considering the ranges of the sums involved in the definition of the $A_t$, one calculates the difference
\[
A_t-A_{t+1}=(-1)^{tn}\mu_P^{\ot_A(2i-t)}\alpha_{2i-t}\left(\psi_{t+1}^{t}-\sum\limits_{s=1}^{t}\psi_s^{t}\phi_{t-s+1}\right)+\sum\limits_{j=0}^{2i-t-1}(-1)^{(2i-j)n}\mu_P^{\ot_Aj}\alpha_j\psi_{2i-t-j}^{2i-j}\phi_{t+1}
\]
\[
=(-1)^{tn}\mu_P^{\ot_A(2i-t)}\alpha_{2i-t}\partial(\phi_{t+1})+\sum\limits_{j=0}^{2i-t-1}(-1)^{(2i-j)n}\mu_P^{\ot_Aj}\alpha_j\psi_{2i-t-j}^{2i-j}\phi_{t+1}
\]
\[
=(-1)^{(t+1)n}\mu_P^{\ot_A(2i-t)}\partial(\alpha_{2i-t}\phi_{t+1})\sim 0 
\]
for $0\le t\le 2i$. Thus, $\mu_P^{\ot_A(2i+1)}[\delta,\alpha]_{2i+1}\sim -A_0\sim -A_{2i+1}=0$.
\end{proof}

Now we are ready to prove existence of
$A_{\infty}$-coderivations corresponding to Hochschild cocycles. 

\begin{theorem}\label{lift_exist}
Let $P$ be any shifted projective bimodule resolution of $A$, with $A_\infty$-coalgebra structure as in Theorem~\ref{Acostruc}. 
For any chain map $f:P\rightarrow A$ of degree $n$, there exists an $A_{\infty}$-coderivation $\alpha_f$ of the same degree with $(\alpha_f)_{0}=f$.
\end{theorem}

\begin{proof}
%For simplicity, we will assume that $\delta$ satisfies the conditions of Corollary \ref{strict}. If $\delta$ does not satisfy these conditions, then we can choose $\tilde\delta$ satisfying them and construct an $A_{\infty}$-coderivation $\tilde\alpha_f$ of $(P,\tilde\delta)$ such that  $(\tilde\alpha_f)_{0}=f$.  Then $[\delta,\tilde\alpha_f]_0=0$ and, by Lemma \ref{InnCoder}, we can choose $\beta\in\DerP$ such that $[\delta,\tilde\alpha_f]=[\delta,\beta]$. Moreover, since $2\nmid 1$, the construction from the proof of the last mentioned lemma allows to construct $\beta$ in such a way that $\beta_0=0$.  Thus, we can set $\alpha=\tilde\alpha-\beta$.

Let us set $(\alpha_f)_{0} = f$, and note that $(\alpha_f \circ \delta)_{0} = f \circ \delta_1 = f d_{P[-1]} =0$
since $f$ is a cocycle. On the other hand, $(\delta\circ \alpha_f)_{0} = \delta_0\circ (\alpha_f)_1 =0$ for any choice of $(\alpha_f)_1$ since $\delta_{0}=0$.

Now assume that $i>0$ and $(\alpha_f)_k$ has been defined for $0\le k\le i-1$ in such a way that $[\delta,\alpha_f]_k=0$ for $0\le k\le i-1$.
We are going to prove that it is possible to define $(\tilde\alpha_f)_k$ for $0\le k\le i$ in such a way that $(\tilde\alpha_f)_k=(\alpha_f)_k$ for $k<i-1$ and $[\delta,\tilde\alpha_f]_k=0$ for $0\le k\le i$ with any choice of $(\tilde\alpha_f)_k$ for $k>i$.
Since $[\delta,[\delta,\alpha]]=0$ for any $\alpha\in\DerP$, we have
$$
0=[\delta,[\delta,\alpha_f]]_i=(\delta\circ[\delta,\alpha_f])_i-(-1)^n([\delta,\alpha_f]\circ\delta)_i=\partial([\delta,\alpha_f]_i).
$$
Note now that $[\delta,\alpha_f]_i=\partial\big((\alpha_f)_i\big)+\Phi_{\alpha_0,\dots,\alpha_{i-1}}$, where $\Phi_{\alpha_0,\dots,\alpha_{i-1}}:P\rightarrow P^{\ot_Ai}$ is a map depending on $(\alpha_f)_k$ with $0\le k\le i-1$. Hence, we have $\partial(\Phi_{\alpha_0,\dots,\alpha_{i-1}})=0$ and $(\alpha_f)_i$ such that $[\delta,\alpha_f]_i=0$ exists if and only if $\mu_P^{\ot_Ai}\Phi_{\alpha_0,\dots,\alpha_{i-1}}\sim 0$.

In the case $2\nmid i$ we have $\mu_P^{\ot_Ai}[\delta,\alpha_f]_i\sim 0$ for any choice of $(\alpha_f)_k$ ($k\ge i$) by Lemma~\ref{main_tec}, that is $\mu_P^{\ot_Ai}\Phi_{\alpha_0,\dots,\alpha_{i-1}}\sim \mu_P^{\ot_Ai}[\delta,\alpha_f]_i\sim 0$.

Suppose now that $2\mid i$. Let us introduce $u=\left(1^{\ot_A(i-1)}\ot_A\mu_P\right)\Phi_{\alpha_0,\dots,\alpha_{i-1}}:P\rightarrow P^{\ot_A(i-1)}$. By construction, we have $\mu_P^{\ot_A(i-1)}u=\mu_P^{\ot_Ai}\Phi_{\alpha_0,\dots,\alpha_{i-1}}$ and $\partial(u)=0$.
Let us set $(\tilde\alpha_f)_k=(\alpha_f)_k$ for $0\le k\le i-2$ and $(\tilde\alpha_f)_{i-1}=(\alpha_f)_{i-1}+u$. We still have $[\delta,\tilde\alpha_f]_k=0$ for $0\le k\le i-1$, since $\partial(\Phi_{\tilde\alpha_0,\dots,\tilde\alpha_{i-1}})=0$. On the other hand, direct calculations show that
\[
\begin{aligned}
&\mu_P^{\ot_Ai}\Phi_{\tilde\alpha_0,\dots,\tilde\alpha_{i-1}}\\
&= \ \mu_P^{\ot_Ai}\Phi_{\alpha_0,\dots,\alpha_{i-1}}+\mu_P^{\ot_Ai}\sum\limits_{j=0}^{i-2}(1^{\ot_Aj}\ot_A\delta_2\ot_A1^{\ot_A(i-2-j)})u
-(-1)^n\mu_P^{\ot_Ai}(u\ot_A1+1\ot_Au)\delta_2\\
&= \ \mu_P^{\ot_A{i-1}}u+\sum\limits_{j=0}^{i-2}(-1)^{i-j}\mu_P^{\ot_A(i-1)}u
-\mu_P^{\ot_A(i-1)}u(1\ot_A\mu_P+(-1)^{i-1}\mu_P\ot_A1)\delta_2\\
&= \ 2\mu_P^{\ot_A(i-1)}u-\mu_P^{\ot_A(i-1)}u(1\ot_A\mu_P-\mu_P\ot_A1)
\delta_2\\
&\sim 2\mu_P^{\ot_A(i-1)}u-2\mu_P^{\ot_A(i-1)}u \ = \ 0.
\end{aligned}
\]
The argument above shows that there is $(\tilde\alpha_f)_i$ such that $[\delta,\tilde\alpha_f]_k=0$ for $0\le k\le i$.
\end{proof}

For some resolutions, the picture is simpler than this general case.
For example, for the bar resolution (Example~\ref{ex:bar}), 
$A_{\infty}$-coderivations $\alpha$ are in fact essentially just
coderivations, that is one may always take $\alpha_1$ to be a
coderivation and $\alpha_i =0$ for $i\geq 2$.
For a resolution $P$ having a coassociative diagonal map
(such as the Koszul resolutions of Section~\ref{sec:Koszul} below),
we may take $\delta_i =0$ for $i\geq 3$, leading to simplified
conditions defining the $\alpha_i$: 
For all $N$,
\begin{equation}\label{eqn:coassoc-delta}
   \partial(\alpha_N) = - \sum_{i+j=N-2} (1^{\ot_A i}\ot_A\delta_2
   \ot_A 1^{\ot_A j})\alpha_{N-1}
   + (-1)^{|\alpha|} (1\ot_A \alpha_{N-1} + \alpha_{N-1}\ot_A 1)\delta_2 .
\end{equation}
Written another way,
\[
\begin{aligned}
 &  \sum_{i+j=N-1} (1^{\ot_A i}\ot_A \delta_1\ot_A 1^{\ot_Aj})\alpha_N
  + \sum_{i+j=N-2} (1^{\ot_Ai}\ot _A \delta_2\ot_A 1^{\ot_A j})\alpha_{N-1}\\
 & \quad \quad = (-1)^{|\alpha|} \alpha_N \delta_1
   + (-1)^{|\alpha|} (1\ot_A \alpha_{N-1} + \alpha_{N-1}\ot_A 1)\delta_2 .
\end{aligned}
\]

\subsection{Homotopy liftings and proof of Theorem~\ref{mainthm}}\label{subsec:proof}

Let us now recall the definition of a homotopy lifting from~\cite{Volkov},
adjusted here to fit our grading and sign choices. 
To prove our main Theorem~\ref{mainthm}, we will need a connection between
$A_{\infty}$-coderivations and homotopy liftings, given by Lemma~\ref{1comp_homlif} below.
First we define homotopy liftings.

\begin{Def}\label{def:hom_lif}
Let $P$ be a shifted projective bimodule resolution of $A$ with $A_{\infty}$-coalgebra
structure $\delta$, and let 
$f:P\rightarrow A$ be a chain map. A map $\phi_f:P\rightarrow P$ is called a {\it homotopy lifting} of the pair $(f,\delta_2)$ if $\partial(\phi_f)=(f\otimes_A 1+1\otimes_A f)\delta_2$ and $\mu_P\phi_f+f\phi\sim 0$ for some $\phi:P\rightarrow P$ of degree $-1$ such that 
\begin{equation}\label{eq:hom1}
\partial(\phi)=(\mu_P\otimes_A 1+1\otimes_A \mu_P)\delta_2 . 
\end{equation}
\end{Def}

\begin{rema}
Note that a map $\phi:P\rightarrow P$ of degree $-1$ satisfying \eqref{eq:hom1} exists~\cite{Volkov}. Moreover, any two such maps are homotopic. Thus, the condition $\mu_P\phi_f+f\phi\sim 0$ holds for some $\phi$ satisfying \eqref{eq:hom1} if and only if it holds for any $\phi$ satisfying \eqref{eq:hom1}.
\end{rema}

\begin{Lemma}\label{1comp_homlif}
Let $P$ be a shifted projective bimodule resolution of $A$ with weakly counital $A_{\infty}$-coalgebra
structure $\delta$.  Let $\alpha$ be an $A_{\infty}$-coderivation on $P$ of degree $n$. Then $(-1)^n\alpha_1$ is a homotopy lifting for $(\alpha_0,\delta_2)$.
\end{Lemma}

\begin{proof} The equality $\partial(\alpha_1) = (-1)^n    (\alpha_0\ot_A 1 + 1\ot _A \alpha_0) \delta_2$ is equivalent to $[\delta,\alpha]_1=0$.  So we need only find $\phi$ solving~\eqref{eq:hom1} and $\mu_P\alpha_1+\alpha_1\phi\sim 0$.

Note now that
\begin{equation}\label{eq:1209}
\begin{aligned}
0=\mu_P^{\ot_A2}[\delta,\alpha]_2&\sim (\mu_P\ot_A\mu_P)\delta_2\alpha_1-(-1)^n\alpha_0\psi_2^2-\mu_P\alpha_1(1\ot_A\mu_P-\mu_P\ot_A 1)\delta_2\\
&=\mu_P\alpha_1-(-1)^n\alpha_0\psi_2^2-\mu_P\alpha_1(1\ot_A\mu_P-\mu_P\ot_A 1)\delta_2 .
\end{aligned}
\end{equation}
There exists $u:P\rightarrow P$ of degree $-1$ such that $\partial(u)=(1\ot_A\mu_P-\mu_P\ot_A 1)\delta_2-2$, by Lemma~\ref{lem:940}. We will show that $\phi=\psi_2^2-(1\ot_A \mu_P + \mu_P\ot _A 1)\delta_2u$ provides the desired function.
\par

We have
\[
\begin{aligned}
\mu_P\alpha_1(1\ot_A\mu_P-\mu_P\ot_A 1)\delta_2&=2\mu_P\alpha_1+\mu_P\alpha_1\partial(u)\\
&\sim 2\mu_P\alpha_1-(-1)^n\alpha_0(1\ot_A \mu_P + \mu_P\ot _A 1)\delta_2u.
\end{aligned}
\]
Thus, one multiplies~\eqref{eq:1209} by $(-1)^{n+1}$ and applies the above equivalence to find
$$
0\sim(-1)^n\mu_P\alpha_1+\alpha_0\big(\psi_2^2-(1\ot_A \mu_P + \mu_P\ot _A 1)\delta_2u\big)=(-1)^n\mu_P\alpha_1+\alpha_0\phi.
$$
Finally, the equality
$$
\partial(\phi)=\partial\big(\psi_2^2-(1\ot_A \mu_P + \mu_P\ot _A 1)\delta_2u\big)=(\mu_P\otimes_A 1+1\otimes_A\mu_P)\delta_2
$$
follows from the expression
\[
\partial(\psi_2^2)=\psi_1^2\psi_1^1=-(\mu_P\ot_A 1+1\ot_A\mu_P)\delta_2+(\mu_P\ot_A 1+1\ot_A\mu_P)\delta_2(1\ot_A\mu_P-\mu_P\ot_A 1)\delta_2
\]
of Lemma~\ref{psider} and the calculation
\[
\begin{aligned}
&-\partial\big((\mu_P\ot _A 1+1\ot_A\mu_P)\delta_2u\big)=-(\mu_P\ot _A 1+1\ot_A\mu_P)\delta_2\partial(u)\\
&\hspace{2cm}=2(\mu_P\ot _A 1+1\ot_A\mu_P)\delta_2-(\mu_P\ot _A 1+1\ot_A\mu_P)\delta_2(1\ot_A\mu_P-\mu_P\ot_A 1)\delta_2.
\end{aligned}
\]
\end{proof}

Now we will prove our main Theorem~\ref{mainthm}.

\begin{proof}[Proof of Theorem~\ref{mainthm}] 
By Theorem~\ref{lift_exist}, for any Hochschild cocycle $f$ on $P$, there exists an $A_{\infty}$-coderivation $\alpha_{f}$ such that $(\alpha_{f})_0=f$. By Lemma \ref{InnCoder}, this construction gives a well-defined bijection $F$ from $\HH^*(A)$ to $\big(\Coder^{\infty}_\mathscr{Q}(P)/\Inn^{\infty}_\mathscr{Q}(P)\big)[1]$.
Since 
$$(\alpha_f\smile\alpha_g)_0=(f\ot_Ag)\delta_2 \ \ \mbox{ and }
 \ \ [\alpha_f,\alpha_g]_0=f(\alpha_g)_1-(-1)^{(|f|-1)(|g|-1)}g(\alpha_f)_1,$$
$F$ is an isomorphism of Gerstenhaber algebras:
The former expression induces precisely the cup product on $\HH^*(A)$.
The latter expression induces the Gerstenhaber bracket on $\HH^*(A)$ by
\cite[Theorem 4]{Volkov} and Lemma \ref{1comp_homlif}.
\end{proof}

\subsection{Coderivations on the tensor coalgebra}\label{subsec:Stasheff}
Next we make a direct comparison of $A_{\infty}$-coderivations
on the bar resolution $B(A)$ of $A$ with 
coderivations on the tensor coalgebra of $A$.
This comparison will clarify how we 
view $A_{\infty}$-coderivations on an arbitrary bimodule
resolution $P$ as a generalization of such coderivations when they are cocycles,
and thus our results as generalizing those of Stasheff~\cite{Stasheff}.

We take $P=B(A)$ here, with $A_{\infty}$-coalgebra structure as in Example~\ref{ex:bar}.
Let $\alpha$ be an $A_{\infty}$-coderivation on $B(A)$
for which $\alpha_n=0$ whenever $n\geq 2$ so that $\alpha = \alpha_0+\alpha_1$.
By Lemma~\ref{1comp_homlif}, up to a sign, $\alpha_1$ is a homotopy lifting
for $(\alpha_0, \delta_2)$.
View $T(A)$ as embedded in $B(A)$ via $A^{\ot n}\cong k\ot A^{\ot n}\ot k\hookrightarrow
A^{\ot (n+2)}$ for each $n$.
Then $\alpha_1|_{T(A)}$ is a coderivation that is in the kernel of the differential on $T(A)$.
Conversely, given a coderivation on $T(A)$ that is a cocycle, 
composing with projection onto $A$
yields an element of $\Hom_k(T(A),A)$ which can be extended to an $A$-bimodule
homomorphism $\alpha_0$ on $B(A)$.
We may extend the original coderivation and denote the corresponding function
to $A$ by $\alpha_1$, and the sum $\alpha_0\pm \alpha_1$ is thus 
an $A_{\infty}$-coderivation on $B(A)$. 
Under this correspondence, brackets are preserved by their
definitions as graded commutators, and so our Theorem~\ref{mainthm}
generalizes Stasheff's description of the Gerstenhaber bracket
as a graded commutator of coderivations on the tensor coalgebra~\cite{Stasheff}.

\subsection{Examples}

We now give some examples.
In Example~\ref{ex:x2} we consider 
a Koszul algebra and its coassociative Koszul resolution,
while in Examples~\ref{ex:x3} and~\ref{ex:x3char3} we consider 
resolutions with nonzero higher $A_{\infty}$-coalgebra maps.

\begin{example}\label{ex:x2} 
Let $\kk$ be an arbitrary field and let A =$\kk [x]/(x^2)$.
Let $P$ be the following resolution of $A$
(cf.\ Example~\ref{ex:xn}):
\[
  P: \hspace{2cm} \cdots \stackrel{ \cdot v}{\longrightarrow} A\ot A 
 \stackrel{ \cdot u}{\longrightarrow} A\ot A 
 \stackrel{ \cdot v}{\longrightarrow} A\ot A 
 \stackrel{ \cdot u}{\longrightarrow} A\ot A , 
\]
where $u = x\ot 1 -1\ot x$ and $v=x\ot 1 + 1\ot x$. 
For each $i$, let $e_i$ denote the element $1\ot 1$ in $A\ot A$.
An embedding of $P$ into the bar resolution on $A$
(defined in Example~\ref{ex:bar}) is given by
\[
  \iota ( e_i) = 1\ot x\ot x\ot\cdots\ot x\ot 1
   \ \ \ \mbox{ in } \ A^{\ot (i+2)} 
\]
for each $i$.
An induced map $\delta_2:P\rightarrow P\ot_A P$ is thus given by
\[
   \delta_2(e_i) = \sum_{j+l=i}(-1)^j e_j\ot e_l 
\]
for all $i$. 
We may take $\delta_n =0$ for $n\geq 3$.
One may check directly that $\delta$ gives $P$ the structure
of an $A_{\infty}$-coalgebra.
See also Section~\ref{sec:Koszul} for a general construction for Koszul algebras.

Now consider the Hochschild 1-cocycle on $A$ 
given by $\alpha_0 (e_1) = x$
(and $\alpha_0(e_i) =0$ for $i\neq 1$).
Let 
\[
   \alpha_1(e_i) = - i e_i
\]
and $\alpha_m(e_i)=0$ for all $i$ and all $m\geq 2$.
Straightforward calculations show that $(\alpha_i)$ is an 
$A_{\infty}$-coderivation.
Consider the Hochschild 2-cocycle given by $\beta_0(e_2)=x$
(and $\beta_0(e_i)=0$ for $i\neq 2$).
Let 
\[
   \beta_1(e_{2i})=-e_{2i-1} \ \ \ \mbox{ and } \ \ \
  \beta_1(e_{2i+1})=0
\]
for all $i$. Let $\beta_m(e_i)=0$ for all $m\geq 2$ and all $i$.
Then $(\beta_i)$ is an $A_{\infty}$-coderivation.
%We find the 2-cocycle $[\beta,\alpha]_0 =\beta_0\alpha_1 -\alpha_0\beta_1 = -\beta_0$ by evaluating on $e_2$.
\end{example}

\begin{example}\label{ex:x3}
Let $\kk$ be an arbitrary field and let $A=\kk[x]/(x^3)$.
We take resolution $P$ with $A_{\infty}$-coalgebra structure
$\delta$ as given in Example~\ref{ex:xn}.
Consider the Hochschild 1-cocycle given by $\alpha_0(e_1)=x$
(and $\alpha_0(e_i)=0$ for $i\neq 1$).
Let
\begin{eqnarray*}
  \alpha_1(e_{2i}) & = & -3i e_{2i} , \\
  \alpha_1(e_{2i+1}) & = & (-3i-1) e_{2i+1}
\end{eqnarray*}
for all $i$ and $\alpha_m(e_i) =0$ for all $i$ and all $m\geq 2$. 
Straightforward calculations show that $(\alpha_i)$ is an 
$A_{\infty}$-coderivation.
Consider the Hochschild 2-cocycle given by $\beta_0(e_2)=x$ 
(and $\beta_0(e_i)=0$ for $i\neq 2$).
Let 
\[
   \beta_1(e_{2i})= - e_{2i-1} \ \ \ \mbox{ and }
 \ \ \ \beta_1(e_{2i+1}) = 0
\]
for all $i$.
Let $\beta_m(e_i)=0$ for all $m\geq 2$ and all $i$.
Then $(\beta_i)$ is an $A_{\infty}$-coderivation.
\end{example}

Next is our first example of an $A_{\infty}$-coderivation
$\alpha$ with $\alpha_2$ nonzero. 
See also Example~\ref{ex:poly-ring}. 

\begin{example}\label{ex:x3char3}
Let $\kk$ be a field of characteristic~3, and $A= \kk [x]/(x^3)$.
We again take resolution $P$ and $A_{\infty}$-coalgebra structure
as in Example~\ref{ex:xn}.
Consider the Hochschild 1-cocycle given by
$\alpha_0(e_1) = x^2$ (and $\alpha_0(e_i)=0$ for $i\neq 1$). 
Let
\begin{eqnarray*}
   \alpha_1(e_{2i}) & = & 0, \\
    \alpha_1(e_{2i+1}) & = & - e_{2i+1}x - xe_{2i+1}, \\
    \alpha_2(e_{2i}) & = & 0 , \\
    \alpha_2(e_{2i+1}) & = & \sum_{j+k = i} e_{2j+1}\ot_A e_{2k+1}
\end{eqnarray*}
for all $i$,
and $\alpha_m(e_i)=0$ for all $m>2$ and all $i$. 
Straightforward calculations and induction on $m$ show that 
$(\alpha_i)$ is an $A_{\infty}$-coderivation.
\end{example}

We end this section with some examples in which Gerstenhaber brackets are
found using these techniques. 

\begin{example} 
Let $\kk$ be an arbitrary field.
First let $A=\kk[x]/(x^2)$, as in Example~\ref{ex:x2}.
By Lemma~\ref{1comp_homlif} and the Gerstenhaber bracket formula
of~\cite{Volkov} that is recalled in the proof of Theorem~\ref{mainthm} above,
the cocycles $\alpha_0, \beta_0$ of Example~\ref{ex:x2} have Gerstenhaber
bracket
\[
    [ \alpha,\beta ]_0 = -\alpha_0\beta_1 + \beta_0\alpha_1 = \beta_0 ,
\]
as may be determined by evaluating on $e_2$.

Next let $A = \kk[x]/(x^3)$, as in Example~\ref{ex:x3}.
The cocycles $\alpha_0,\beta_0$ of Example~\ref{ex:x3} similarly
have Gerstenhaber bracket $[\alpha,\beta]_0 = \beta_0$.
\end{example}

\section{Connected graded algebras}\label{sec:conn-gr}

In this section, we will show that many connected graded algebras
have the property that the $A_{\infty}$-coderivations on $P$ corresponding
to Hochschild cocycles are locally finite, that is, they
factor through the direct sum $\oplus_{n\geq 0} P^{\ot_A n}$.  
These are the connected graded algebras satisfying a certain finiteness condition, (HF) below, which is always satisfied by Koszul algebras and Noetherian connected graded algebras.  We will also give some explicit formulas and further results on
$A_{\infty}$-coderivations for Koszul algebras. 

\subsection{Cohomology and connected graded algebras}
\label{sect:Hgr}

Let $A$ be a connected graded $\kk$-algebra, so that $A=\oplus_{n\geq 0}A_n$ with $A_0\cong \kk$.  We suppose additionally that $A$ satisfies the following homological finiteness condition:
\begin{enumerate}
\item[(HF)] $\dim \Ext_A^i(\kk,\kk)<\infty$ for each $i$.
\end{enumerate}
We note that any connected graded algebra which satisfies (HF) must be finitely generated, as $\Ext^1_A(\kk,\kk)$ is identified with the dual of a minimal space of generators (see e.g.~\cite[Section~5]{LPWZ}).  We also note that if $A$ is Noetherian, or Koszul, then $A$ satisfies the condition (HF).

Let $B$ be a connected graded algebra.  We are thinking here of $B=A$ or $B=A^e$.  Then any finitely generated graded $B$-module $M$ admits a minimal (graded) projective resolution $P(M)\to M$, with $P(M)_{-i}=B\ot V_{-i}$ for some graded vector space $V_{-i}$.  The minimality assumption gives $\Ext_B^i(M,\kk)=(V_{-i})^\ast$ (cf.~\cite[Section~2]{atv}).  We refer to the gradings on $P(M)$ induced by the grading on $B$ as the {\it internal grading}.  When $P(M)$ is finitely generated in each degree, we also refer to the induced grading on each $\Ext_B^i(M,\kk)$ as the internal grading.
\par

Recall that for any augmented algebra $A$ we have $\Ext^\ast_{A^e}(A,\kk)\cong\Ext_A^\ast(\kk,\kk)$~\cite[Lemma~9.1.9]{weibel}.  In particular, for our connected graded algebra $A$ the condition (HF) implies $\dim \Ext^i_{A^e}(A,\kk)<\infty$ for each $i$, and hence the minimal bimodule resolution $P(A)\to A$ of $A$ is such that each $P(A)_{-i}$ is finitely generated.
\par

The important point here is that any $A$ satisfying (HF) admits a graded bimodule resolution $P'$ with the following property:
\begin{enumerate}
\item[(HF$'$)] Each $P'_{-i}$ is finitely generated, in internal degrees $i\le d\le l(i)$, for some positive integer $l(i)$ which depends on $i$.
\end{enumerate}
(The lower bound $i$ is obtained by induction and the fact that $A$ is generated in positive degree.)  Indeed, the existence of a graded bimodule resolution $P'$ satisfying (HF$'$) is equivalent to condition (HF).  We note that such a $P'$ is free in each degree, by a Nakayama type argument.  We say a shifted bimodule resolution $P$ satisfies (HF$'$) if the associated non-shifted bimodule resolution $P'$ satisfies (HF$'$).

\subsection{Local finiteness of $A_{\infty}$-coderivations}

Let $P$ be a graded, shifted, bimodule resolution of $A$.  Then $T(P)$ has an induced internal grading, and we may speak of homogenous endomorphisms of $T(P)$ with respect to the internal grading (of some particular degree).  For the completed algebra $\widehat{T}(P)$, we say a continuous endomorphism of $\widehat{T}(P)$ is homogenous with respect to the internal grading if each of the truncations $\widehat{T}(P)/I^{n'}\to \widehat{T}(P)/I^n$ is homogenous, where $I$ denotes the ideal generated by $P$.
\par

We fix now $A$ a connected graded algebra satisfying the finiteness condition (HF), and $P$ a shifted bimodule resolution satisfying (HF$'$).  Recall that elements in $\mrm{Der}(\Omega)$, for a topological algebra $\Omega$, are by definition continuous.

\begin{lemma}\label{lem:lf}
Any degree $n$ derivation $f\in \mrm{Der}_\msc{Q}(\widehat{T}(P))$ which is homogenous with respect to the internal grading is locally finite.  Rather, any such $f$ is in the image of the completion map $\mrm{Der}_\msc{Q}(T(P))\to \mrm{Der}_\msc{Q}(\widehat{T}(P))$.
\end{lemma}

\begin{proof}
Let $d$ denote the degree of $f$ with respect to the internal grading.  It suffices to show that for each $i$ there is a corresponding $N$ such that $f_k:P\to P^{\ot_A k}$ vanishes on $P_{-i+1}$ whenever $k>N$.  Let $l$ be the maximal internal degree of a generator for $P_{-i+1}$, i.e.\ the maximal nonvanishing degree of the reduction $k\ot_{A^e}P_{-i+1}$.  The restriction of $f_k$ to $P_{-i+1}$ corresponds to a degree $n-k+1$ map from the unshifted resolution $P_{-i}'\to (P')^{\ot_A k}$, which we also denote by $f_k$ by abuse of notation.  Since each $P'_{-j}$ is generated in internal degrees at least $j$, the space $\left((P')^{\ot_A k}\right)_{-i+n-k+1}$ vanishes in degrees less than $k+i-n-1$.  We can therefore take $N=d+l-i+n+1$ to find that $f_k|_{P'_{-i}}$ vanishes for all $k>N$.  Thus $f$ is locally finite.
\par

The local finiteness implies that $f|_P:P\to\widehat{T}(P)$ has image in $T(P)$, and thus defines a unique derivation $F:T(P)\to T(P)$ such that the completion $\hat{F}$ recovers $f$.
\end{proof}

In our construction of the $A_\infty$-coalgebra structure on $P$, in the proof of Proposition~\ref{Acostruc}, we are free to assume that all of the $\delta_i$ preserve the internal grading.  By the previous lemma, the $A_\infty$-coalgebra structure on $P$ will therefore be locally finite, and we may consider the discrete cobar construction $\Omega^\mrm{disc}P$ of Definition~\ref{def:discrete-cobar}.  
We have the completion map
\begin{equation}\label{eq:cmplt}
cmpl:\mrm{Der}_\msc{Q}(\Omega^\mrm{disc}P)\to \mrm{Der}_\msc{Q}(\Omega P),\ \ f\mapsto \hat{f},
\end{equation}
which is a map of dg Lie algebras.  The completion map is also a map of dg algebras, under the convolution product of Section~\ref{subsec:cupprod}.
\par

Let us take $\mrm{Coder}_{lf}^\infty(P)$ and $\mrm{Inn}^\infty_{lf}(P)$ to be the cocycles and coboundaries in $\mrm{Der}_\msc{Q}(\Omega^\mrm{disc}P)$, respectively.

\begin{theorem}\label{thm:lf}
For $A$ a connected graded algebra satisfying {\rm (HF)}, and $P$ a shifted resolution satisfying {\rm(HF$'$)}, the completion map~\eqref{eq:cmplt} is a quasi-isomorphism.  Consequently, there is an identification of Gerstenhaber algebras
\begin{equation}\label{eq:1461}
\HH^\ast(A)\cong \mrm{Coder}^\infty_{lf}(P)/\mrm{Inn}^\infty_{lf}(P)[1].
\end{equation}
\end{theorem}

\begin{proof}
We have the diagram
\[
\xymatrix{
\mrm{Der}_\msc{Q}(\Omega^\mrm{disc}P)\ar[rr]^{cmpl}\ar[dr]_{z_{lf}} & & \mrm{Der}_\msc{Q}(\Omega P)\ar[dl]^z\\
 &\Hom_\msc{Q}(P,A) &,
}
\]
where the maps to $\Hom_\msc{Q}(P,A)$ take $f$ to $f_0$.  The map $z$ is a quasi-isomorphism by Theorem~\ref{mainthm}.  (Or, more precisely, by Lemma~\ref{InnCoder} and Theorem~\ref{lift_exist}.)  To see that $cmpl$ is a quasi-isomorphism, it suffices to show that $z_{lf}$ is a quasi-isomorphism.  In this case, the identification~\eqref{eq:1461} will follow by Theorem~\ref{mainthm} and the fact that $cmpl$ is a dg (Lie) algebra map.
\par

To see that $z_{lf}$ is a quasi-isomorphism, we note that $\Hom_\msc{Q}^n(P,A)$ is spanned by homogeneous elements with respect to the internal grading, and for homogenous $\alpha_0\in \Hom_\msc{Q}^n(P,A)$ we can produce a lift $\alpha\in \mrm{Der}_\msc{Q}(\Omega P)$ which is also homogenous with respect to the internal degree.  Indeed, our construction of such a lift in the proof of Theorem~\ref{lift_exist} can be done in a homogeneous manner.  In this case $\alpha$ is locally finite, by Lemma~\ref{lem:lf}, and we therefore have a lift of $\alpha_0$ in $\mrm{Coder}^\infty_{lf}(P)$.  This shows that any cocycle in $\Hom_\msc{Q}(P,A)$ lifts to a cocycle in $\mrm{Der}_\msc{Q}(\Omega^\mrm{disc}P)$.
\par

To see that $z_{lf}$ is a quasi-isomorphism we must show that if $\alpha$ is a locally finite $A_{\infty}$-coderivation such that $\alpha_0$ is a coboundary, then $\alpha$ admits a locally finite bounding element $\beta$.
\par

Since each $P_{-i}$ is generated in finitely many internal degrees, any $f\in \mrm{Der}_{\mathscr{Q}}(\Omega P)$ can be written uniquely as a convergent sum $\sum_{k=-\infty}^\infty f(k)$ such that $f(k)$ is homogenous of internal degree $k$ and the composites $\pi_nf(k):P_{-i}\to \oplus_{l\le n}P^{\ot l}$ vanish for all but finitely many $k$.  We have $[\delta,\sum_k f(k)]=\sum_k[\delta,f(k)]$, and $[\delta,f(k)]=[\delta,f](k)$ since $\delta$ is homogenous of degree $0$.  So, an element $\alpha$ is bounded by some $\beta$ if and only if $\alpha(k)=[\delta,\beta(k)]$ for each $k$.
\par

Now, let $P(i)$ denote the homogeneous degree $i$ portion of $P$ with respect to the internal degree.  Since $P_{-i+1}$ is generated in a finite number of degrees which are at least $i$, we have $P(i)=\oplus_{l\le i+1} P_l(i)$, and one can conclude from Lemma~\ref{lem:lf} that a derivation $f$ is locally finite if and only if for each $i$ there is an integer $N(i)$ such that $f|_{P(i)}=\sum_{k=-N(i)}^{k=N(i)}f(k)|_{P(i)}$.  For a locally finite cocycle $\alpha$, we have $\alpha|_{P(i)}=\sum_{k=-M(i)}^{M(i)}\alpha(k)|_{P(i)}$ for some integers $M(i)$, and we claim that if $\alpha_0$ is a coboundary then there exists a bounding element $\beta$ for $\alpha$ such that
\[
\beta|_{P(i)}=\sum_{k=-N(i)}^{N(i)}\beta(k)|_{P(i)},\ \ \text{for }N(i)=\mrm{max}\{M(j):0\le j\le i\}.
\]
Indeed, one can use homogeneity of $\delta$ and freeness of $P$ to prove that such a $\beta=\sum_{n=1}^\infty\beta_n$ exists via induction on $n$ (cf.\ the proof of Lemma~\ref{InnCoder}).  Such a $\beta$ provides the desired locally finite bounding element.
\end{proof}

We next give a small example illustrating local finiteness.

\begin{example}\label{ex:poly-ring} 
Let $\kk$ be a field of arbitrary characteristic, and $A=\kk[x]$.
Let $K$ be the resolution of $A$ as an $A^e$-module given by
\[
   K: \hspace{2cm}   0\rightarrow A^e 
    \stackrel{\cdot (x\ot 1 - 1\ot x)}{\relbar\joinrel\relbar\joinrel
     \relbar\joinrel\relbar\joinrel\relbar\joinrel\longrightarrow}
        A^e .
\]
Then $K$ embeds into the bar resolution: 
Identify the degree~0 component $K_0 = A^e$ with that in the bar 
resolution~(\ref{ex:bar}), and in degree~1, send $e_1:= 1\ot 1$ in $K_1=A^e$
to $1\ot x\ot 1$ in $A^{\ot 3}$ in the bar resolution.
The coalgebra structure $\delta_2$ on the bar resolution restricts to
a coalgebra structure for $K$.
We find that
\begin{eqnarray*}
  \delta_2(e_0) & = & e_0\ot e_0 , \\
  \delta_2(e_1)& = &e_0\ot e_1 - e_1\ot e_0 .
\end{eqnarray*}
Let $n$ be any positive integer and consider the Hochschild 1-cocycle given by 
$\alpha_0(e_1)=x^n$ (and $\alpha_0(e_0)=0$).
Let 
\begin{eqnarray*}
\alpha_1(e_1) & = &  - \sum_{a+b = n-1} x^a e_1 x^b , \\
\alpha_2(e_1) & = & \sum_{a+b+c=n-2} x^a e_1\ot_A x^b e_1 x^c , \\
   \vdots     & &  \vdots  \\
\alpha_i(e_1)  & = & (-1)^i \sum_{a_1+\cdots a_{i+1}=n-i} x^{a_1}e_1
     \ot_A x^{a_2}e_1\ot_A \cdots\ot_A x^{a_i} e_1 x^{a_{i+1}} , \\
   \vdots  &&  \vdots \\
\alpha_n(e_1) &= & (-1)^n e_1\ot_A\cdots\ot_A e_1 , 
\end{eqnarray*}
$\alpha_m(e_0)=0$ for all $m$, 
and $\alpha_m(e_1)=0$ for all $m>n$. 
Lengthy calculations show that $(\alpha_i)$ is an $A_{\infty}$-coderivation;
see also the general formula~(\ref{eqn:general-Koszul}) below 
for Koszul algebras.
\end{example}

\subsection{Koszul algebras}\label{sec:Koszul}

Now we focus on Koszul algebras and 
show that the grading on such algebras appears in the structure
of $A_{\infty}$-coderivations corresponding to Hochschild cocycles.
These $A_{\infty}$-coderivations will all be locally finite,
that is, they take values in $\oplus_{k\geq 0} P^{\ot_A k}$,
by Lemma~\ref{lem:lf}. 
We recall  definitions of Koszul algebras and their resolutions next. 
Our Koszul algebras $A$ will always be graded and connected.
As discussed above, the grading on $A$ imposes an additional grading on
$\Ext^*_A(\kk,\kk)$ so that for each $i$, 
$\Ext^i_A(\kk,\kk)=\oplus_j \Ext^{i,j}_A(\kk,\kk)$.

By definition, $A$ is a {\em Koszul algebra} 
if $\Ext^{i,j}_A(\kk,\kk)=0$ whenever $i\neq j$.
This implies that $A$ is quadratic, that is all its relations are
in degree~2.
It is equivalent to exactness of the following sequence, so that it
is a projective $A^e$-module resolution of $A$~\cite{Kraehmer}:
Let $V = A_1$ and let $R \subset V\ot V$ be the subspace of 
quadratic relations, where $\ot = \ot_{\kk}$, 
so that $A\cong T_{\kk}(V)/(R)$.
The Koszul resolution is defined from this information:
\[
  K: \hspace{2cm}   \cdots \longrightarrow A\ot R\ot A\longrightarrow A\ot V\ot A
  \longrightarrow A\ot A
\]
with $K_0=A\ot A$, $K_1=A\ot V\ot A$, and for each $i\geq 2$,
$K_i=A\ot K_i'\ot A$ where
\[
   K_i' = \bigcap_{j+l=i-2} (V^{\ot j}\ot R\ot V^{\ot l} ) .
\]
Each $K_i$ embeds canonically into $A^{\ot (i+2)}$, and $K$
is in this way a subcomplex of the bar resolution of Example~\ref{ex:bar}.
It is shown in~\cite{BGSS} that through this embedding, there is a
coassociative diagonal map $\Delta_2$ on $K$.
This map may be modified by including some signs 
as in Example~\ref{ex:bar} to obtain a map $\delta_2$
on $K$: 
\[
   \delta_2(a_0\ot\cdots \ot a_{n+1}) = \sum_{i=0}^n (-1)^i
      (a_0\ot\cdots \ot a_i\ot 1)\ot _A (1\ot a_{i+1}\ot\cdots\ot a_{n+1})
\]
for all $a_0,\ldots, a_{n+1}\in A$ on the bar resolution (it may be checked that
$\delta_2$ in fact restricts to a map from $K_n$ to $K_{n-1}$). 
Take $\delta_i=0$ for all $i\geq 3$.

We set $P_i:=K_{1-i}$, so that $P$ is a shifted projective resolution
of $A$ as an $A^e$-module. 
We will use the $A_{\infty}$-coalgebra structure on $P$ as given by
the maps $\delta_n$  above.
In particular, we have $P_i=0$ and $\delta_i=0$ for $i>2$.

%Let us recall the remaining part of the construction of $K$. As it was mentioned above, we have $A=T_{\kk}(V)/I$, where $V$ is a $\kk$-linear space and $I$ is the ideal generated by the subspace $R\subset V\ot V$.
%We define $K_i'$ and $K_i$ as in Example \ref{ex:Koszul}. The differential $d_K:K\rightarrow K$ is defined as follows. Two canonical inclusions $\iota_{1,i}:K_{i+1}'\rightarrow V\ot K_i'$ and $\iota_{i,1}:K_{i+1}'\rightarrow K_i'\ot V$ induce two maps from $K_{i+1}$ to $K_i$.\footnote{S: I think we need to choose a basis of $V$ for this? Check
%the rest of this paragraph and the next.} 
%Let us define $d_K=d_1+d_2$, where $d_{1,i},(-1)^{i+1}d_{2,i}:K_{i+1}\rightarrow K_i$ are the maps induced by $\iota_{1,i}$ and $\iota_{i,1}$ respectively.

%Consider now the canonical inclusions $\iota_{i,j}:K_{i+j}'\rightarrow K_i'\ot K_j'$. The $r$-th component of the map $\Delta_2:K\rightarrow K\ot_A K$ is $\sum\limits_{i=0}^rd_{i,r-i}:K_{r}\rightarrow \oplus_{i=0}^rK_i\ot_A K_{r-i}$, where $d_{i,i-r}$ is the map induced by $\iota_{i,i-r}$. To obtain $\delta_2$ one has to replace $d_{i,r-i}$ by $(-1)^id_{i,r-i}$ in the definition of $\Delta_2$. {\color{green}I haven't checked this carefully!}

Note that the grading on $A$ induces a grading on $P$. Thus, $P$ has a $\bZ\times \bZ$-grading such that elements from $A_p\ot K_i'\ot A_q$ have degree $(i,p+i+q)$. 
A homogeneous element with respect to this grading can be represented by a map from $K_n'$ to $A_p$ for some integers $n$ and $p$. The corresponding element has degree $(n, p-n)$ and belongs to $\HH^n(A)$.

Let us take an element of $\HH^*(A)$ that is represented by a map from $K_n'$ to $A_p$ and construct $\alpha\in \Coder^{\infty}_\mathscr{Q}(P)$ representing it
following the prescription in the proof of Theorem~\ref{lift_exist}. 
Start with $\alpha_0$, a map from $P_{1-n}$ to $A$.
For each $k$, 
we construct $\alpha_k|_{P_{-i}}$ by induction on $i$.
%%%%%%%%%%
%That is, $\alpha_k(K_i')$ is a subset of the $(p+i-n)$-th component of $(P^{\ot_Ak})_{n-i}$.
%%%%%%%%%%
For each $s\geq 0$, denote by $\pi_s : V^{\ot s}\rightarrow A_s$ the
canonical projection and let $\tau: V\rightarrow K_1$ be the 
canonical embedding. 

Let us set $\alpha_k|_{P_{-i}}=0$ for $i<n-1$. Let us now choose some vector space basis $x_1,\dots,x_l$ of $K_n'$. Represent $\alpha_0(x_t)$ by an element $y_t\in V^{\ot p}$, via a choice of
section $A_p\rightarrow V^{\ot p}$ of $\pi_p$.
We will use the canonical identification of $(K_1\ot A)\ot_A (A\ot K_1)$ with $K_1\ot A\ot K_1$.
Let us define $\alpha_k|_{P_{1-n}}$ by the equalities
\begin{equation}\label{eqn:general-Koszul}
\alpha_k(1\ot x_t\ot 1)=%\sum\limits_{\scriptsize\begin{array}{c}i_0,\dots,i_k\ge 0\\\sum_{r=0}^ki_r=p-k\end{array}} 
\sum_{i_0+\cdots + i_k = p-k}
\! (-1)^k(\pi_{i_0}\ot_A \tau\ot_A \pi_{i_1}\ot_A \tau\ot_A\cdots\ot_A \tau\ot_A \pi_{i_{k}})y_t
\end{equation}
for $1\le t\le l$. 

\begin{prop}
Let $A$ be a Koszul algebra, with shifted Koszul resolution $P$, and let
$\alpha_0: P_{1-n}\rightarrow A$ be a Hochschild $n$-cocycle.
Then $\alpha_k$ ($k\ge 1$), as given above on $P_{-i}$ for $i\leq n-1$,
can be extended to components $P_{-i}$ with $i>n-1$ in such a way that $\alpha=\sum_{k\geq 0}\alpha_k$ is an $A_\infty$-coderivation.
\end{prop}

\begin{proof}
As above, we first define $\alpha_k \mid_{P_{-i}} =0$ for $i< n-1$
and $\alpha_k\mid _{P_{1-n}}$ by~(\ref{eqn:general-Koszul}).
We will check that such $\alpha_k$ satisfies the equality
$$\partial(\alpha_k)|_{P_{-i}}=\left.\left((-1)^{n+1}(\alpha_{k-1}\otimes_A 1+1\otimes_A\alpha_{k-1})\delta_2-\sum\limits_{i=0}^{k-2}(1^{\ot_Ai}\ot_A\delta_2\ot_A1^{\ot_A(k-2-i)})\alpha_{k-1}\right)\right|_{P_{-i}}$$
for $i\le n-1$. 
Indeed, 
the left hand side evaluated on $x_t$ in $K_n'$ (that is, on $1\ot x_t\ot 1$ in $P_{1-n}$) is 
\[
\begin{aligned}
 & \sum_{i+j = k-1} (1^{\ot _A i}\ot_A\delta_1 \ot_A 1^{\ot_A j}) \alpha_k(1\ot x_t\ot 1)
       + (-1)^n \alpha_k \delta_1(1\ot x_t\ot 1)\\
 &=\sum_{i+j=k-1}  (1^{\ot_A i}\ot_A \delta_1\ot_A 1^{\ot_A j}) \sum_{i_0+\cdots + i_k=p-k}
   (-1)^k (\pi_{i_0}\ot_A\tau\ot_A\cdots\ot_A\tau\ot_A \pi_{i_k}) y_t\\ 
   &\hspace{4cm}+ (-1)^n \alpha_k(x_t\ot 1 + (-1)^n\ot x_t) .
\end{aligned}
\]
Now $\alpha_k|_{K_{n-1}} = 0$, so $ \alpha_k(x_t\ot 1) = 0 = \alpha_k( 1\ot x_t)$.
The above expression may now be rewritten as 
\[
   \sum\limits_{\scriptsize\begin{array}{c}i_0+\cdots + i_k = p-k\\ j\in\{ 0,\ldots, k-1\}\end{array}}
    (-1)^k (\pi_{i_0}\ot_A\tau_{i_0}\ot_A\cdots \ot_A\delta_1\tau_{i_j}\ot_A\cdots\ot_A
    \pi_{i_k} ) y_t 
\]
where $\tau_{i_l} = \tau$ for all $l$.
The right hand side evaluated on $x_t$ in $K_n'$ is 
\[
\begin{aligned}
 & (-1)^{n+1} (\alpha_{k-1}\ot_A 1 + 1\ot_A \alpha_{k-1})
  ((1\ot 1)\ot_A (1\ot x_t\ot 1) + (-1)^n (1\ot x_t\ot 1)\ot_A (1\ot 1))\\
& - \sum_{i=0}^{k-2} (1^{\ot _A i}\ot_A \delta_2\ot_A 1^{\ot_A (k-2-i)})
   \! \sum_{i_0+\cdots + i_{k-1} = p-k+1}  (-1)^{k-1}
   (\pi_{i_0} \ot_A\tau\ot_A\cdots\ot_A \tau\ot_A \pi_{i_{k-1}}) y_t \\
& = \sum_{i_0+\cdot + i_{k-1}=p-k+1} 
   (-1)^k (\pi_{i_0}\ot_A \tau\ot_A\cdots\ot_A\tau\ot_A\pi_{i_{k-1}}) y_t\ot_A (1\ot 1)\\
   &\hspace{.4cm} + \sum_{i_0+\cdots + i_{k-1}=p-k+1} (-1)^{k-1} (1\ot 1)\ot_A (\pi_{i_0}\ot_A
   \tau\ot_A\cdots\ot_A \tau\ot_A\pi_{i_{k-1}})y_t \\
 & \hspace{.4cm}+ \sum_{\substack{i_0+\cdots + i_{k-1} =p-k+1\\ j\in\{0,\ldots,k-2\}}}
    (-1)^k (\pi_{i_0}\ot_A \tau_{i_0}\ot_A\cdots\ot_A \delta_2\tau_{i_j}\ot_A
  \cdots\ot_A \pi_{i_{k-1}}) y_t  . 
\end{aligned}
\]
The first two sums above also appear within the third, with opposite signs,
and so they cancel.
Comparing the left hand side obtained previously 
with what remains of the right hand side here, since
\begin{eqnarray*}
  \delta_1(1\ot v\ot 1) & = & v\ot 1 - 1\ot v \ \ \mbox{ and } \\
  \delta_2(1\ot v\ot 1) & = &  (1\ot 1)\ot_A (1\ot v\ot 1)
    - (1\ot v\ot 1)\ot_A (1\ot 1) , 
\end{eqnarray*}
each is a telescoping sum in which there is further cancellation,
with remaining terms coinciding. 
Thus we have shown that $\alpha_k$ satisfies, on such $P_{-i}$, 
the condition to be the 
$k$-th component of an $A_{\infty}$-coderivation.
\end{proof}

%{\color{blue}[I don't understand the rest of this. -Cris]}

We use these observations in the next result.

\begin{prop}\label{prop:Pleq0}
Let $A$ be a Koszul algebra, with shifted Koszul resolution $P$,
and let $\alpha_0$ be a Hochschild cocycle mapping $P_{1-n}$ to $A_p$. 
The corresponding $A_{\infty}$-coderivation 
$\alpha$ defined by~(\ref{eqn:general-Koszul}) satisfies 
$\alpha_k(P)\subset P_{\le 0}^{\ot_Ak}$ for all $k\ge 1$.
\end{prop}
\begin{proof}
For $k=1$ the required condition is satisfied automatically:
Specifically, we may define $\alpha_1$ via~(\ref{eqn:general-Koszul}) 
on $P_{-i}$ for $i\le n-1$ as above and extend to $P_{-i}$
for $i>n-1$ where degree conditions force its image to be in $P_{\le 0}$. 
On the $k$-th step we need to construct $\alpha_k$ in such a way that
$$\partial(\alpha_k)=(-1)^{n+1}(\alpha_{k-1}\otimes 1+1\otimes\alpha_{k-1})\delta_2-\sum\limits_{i=0}^{k-2}(1^{\ot_Ai}\ot_A\delta_2\ot_A1^{\ot_A(k-2-i)})\alpha_{k-1}.$$
Assume that $\alpha_{k-1}(P)\subset P_{\le 0}^{\ot_A(k-1)}$. Then it is easy to show using the formulas for $\delta_2$ that the right hand side of the last equality lies in $P_{\le 0}^{\ot_Ak}$. Since we have already constructed $\alpha_k|_{P_{1-n}}$ in such a way that $\alpha_k(P_{1-n})\subset \left(P_{\le 0}^{\ot_Ak}\right)_0$ and $P_{\le 0}^{\ot_Ak}$ is exact in negative degrees, we can produce $\alpha_k$ in such a way that $\alpha_k(P)\subset P_{\le 0}^{\ot_Ak}$.

Note that if $2\nmid k$, then during the construction of $\alpha_{k+1}$ as
in Theorem~\ref{lift_exist}, we replace $\alpha_k$ by
$$
\alpha_k+\left(1^{\ot_Ak}\ot_A\mu_P\right)\left(\sum\limits_{i=0}^{k-1}(1^{\ot_Ai}\ot_A\delta_2\ot_A1^{\ot_A(k-1-i)})\alpha_{k}+(-1)^n(\alpha_{k}\otimes_A 1+1\otimes_A\alpha_{k})\delta_2\right),
$$
but one can show that this does not change our conditions.
\end{proof}

\begin{rema}
By Proposition~\ref{prop:Pleq0} for Koszul algebras, we obtain 
an element $\alpha\in \Coder^{\infty}_\mathscr{Q}(P)$ such that $\alpha_k(P)\subset P_{\le 0}^{\ot_Ak}$ and all the components of $\alpha$ are homogeneous of degree $p-n$ as $A$-bimodule homomorphisms.
In particular, we have $\alpha_k=0$ for $k>p$ and $\alpha_k|_{P_{-i}}=0$ for $i<n-1$. For an element of $\HH^*(A)$ represented by a map from $K_n'$ to $A_p$ then, we have constructed maps $\alpha_p|_{\oplus_{i\ge n}K_i'}:\oplus_{i\ge n}K_i'\rightarrow \left(\oplus_{i\ge 1}K_i'\right)^{\ot p}$. The necessary inclusion $\alpha_p(\oplus_{i\ge n}K_i')\subset \left(\oplus_{i\ge 1}K_i'\right)^{\ot p}$ follows from a degree argument.
Note that in fact $\oplus_{i\ge 0}K_i'$ has the structure of dg coalgebra $C$, called the Koszul dual coalgebra of $A$. The map $\alpha$ induces a map from $C$ to the $p$-th component of the bar resolution of $C$. One can show that this map is a cocycle and that this cocycle is a coboundary if and only if $\alpha_0$ is a coboundary. Moreover, if $\beta\in \Coder^{\infty}_\mathscr{Q}(P)$ is another coderivation corresponding to an element of $\HH^*(A)$ represented by a map from $K_m'$ to $A_p$, then the formula for the bracket $[\alpha,\beta]_{p+q}$ is exactly the formula for the Gerstenhaber bracket of the elements $\alpha_p$ and $\beta_q$. This gives a new proof of the fact that $\HH^*(A)$ and $\HH^*(C)$ are isomorphic as Gerstenhaber algebras (cf.~\cite{Keller-derived}). We do not give the details of the proof here because this fact is well known.
\end{rema}

\bibliographystyle{abbrv}
%\bibliography{infty}

\end{document}